\documentclass[11pt]{article}
\usepackage{amsmath}
\usepackage{amstext,amsbsy}
\usepackage{epsfig,multicol}
\usepackage{amsthm}
\usepackage{amssymb,latexsym}
\usepackage{color}

\theoremstyle{plain}
\newtheorem{theorem}{Theorem}[section]
\newtheorem{lemma}[theorem]{Lemma}
\newtheorem{corollary}[theorem]{Corollary}
\newtheorem{proposition}[theorem]{Proposition}
\newtheorem{example}[theorem]{Example}

\theoremstyle{definition}
\newtheorem{definition}[theorem]{Definition}

\newtheorem{remark}[theorem]{Remark}
\newtheorem*{properties*}{Properties}

\newenvironment{definition*}[1][Definition]{\begin{trivlist}
\item[\hskip \labelsep {\bfseries #1}]}{\end{trivlist}}

\numberwithin{equation}{section}
\newcommand{\m}{\mathrm{maj}}
\newcommand{\se}{\mathrm{ne}}
\newcommand{\de}{\mathrm{des}}
\newcommand{\ci}{c_i}
\newcommand{\cj}{c_j}
\newcommand{\aci}{|c_i|}
\newcommand{\acj}{|c_j|}
\newcommand{\s}{\textbf{s}}
\newcommand{\n}{\textbf{F}}
\newcommand{\pone}{ptr_1}
\newcommand{\ptwo}{ptr_2}
\newcommand{\cone}{C_1}
\newcommand{\ctwo}{C_2}
\newcommand{\R}{\mathcal{R}}
\newcommand{\ds}{\text{desseq}}
\newcommand{\M}{\mathcal{M}}
\newcommand{\F}{\mathcal{F}}
\newcommand{\N}{\mathcal{N}}

\begin{document}

\title{Major Index for 01-Fillings of Moon Polyominoes}
\author{
William Y.C. Chen$^{a,1}$, Svetlana~Poznanovi\'{c}$^{b}$, 
Catherine H. Yan$^{a,b,2}$ and \\
Arthur L.B. Yang$^{a,1}$
 \vspace{.2cm} \\
$^{a}$Center for Combinatorics, LPMC-TJKLC\\
Nankai University, Tianjin 300071, P. R. China
\vspace{.2cm} \\
$^{b}$Department of Mathematics\\
Texas A\&M University, College Station, TX 77843 } 
\date{}
\maketitle
\begin{abstract}
We propose a major index statistic on 01-fillings of moon polyominoes which,
when specialized to certain shapes, reduces to the major index for
permutations and set partitions. We consider the set $\n(\M, \s; A)$ of all
01-fillings of a moon polyomino $\M$ with given column sum $\s$ whose empty rows are 
$A$, and prove that this major index has the same
distribution as the number of north-east
 chains, which are the natural extension
of inversions (resp. crossings) for permutations (resp. set partitions). 
Hence our result generalizes the classical equidistribution results for 
the permutation statistics inv and maj. Two proofs are presented. The first is an 
algebraic one using generating functions, and the second is a bijection 
on 01-fillings of moon polyominoes in the spirit of 
Foata's second fundamental transformation on words and permutations.  
\end{abstract}



{\renewcommand{\thefootnote}{}
\footnote{\emph{E-mail addresses}: chen@nankai.edu.cn (W.Y.C.~Chen), 
spoznan@math.tamu.edu (S.~Poznanovi\'{c}), 
cyan@math.tamu.edu (C.H.~Yan), 
yang@nankai.edu.cn (A.L.B.~Yang).} } 
\footnotetext[1]{The first and the fourth authors were
supported by the 973 Project on
Mathematical Mechanization, the PCSIRT Project of
the Ministry of Education, the Ministry
of Science and Technology, and the National Science Foundation of
China.}

\footnotetext[2]{The third author was supported in part by NSF
grant \#DMS-0653846.}

\section{Introduction} \label{S:introduction}

Given a multiset $S$ of $n$ positive integers, a word on $S$ is a sequence $w=w_1w_2 \dots w_n$ that reorders the elements in $S$.
When $S= [n]:=\{1,\dots,n\}$  the word is  a permutation. 
A pair $(w_i, w_j)$ is called an \emph{inversion} of $w$ if $i<j$ and $w_i >w_j$. 
One well-known statistic on words and permutations is $\textrm{inv}(w)$, defined as 
the number of inversions of $w$. 
The \textit{descent set} and \textit{descent statistic} of a word $w$ are defined as
\[\text{Des}(w)=\{i:1 \leq i \leq n-1, w_i>w_{i+1}\}, \hspace{1cm} \text{des}(w)=\# \text{Des}(w).\]
In \cite{Mac} P.~MacMahon defined the \textit{major index} statistic for a word $w$ as
\[ \text{maj}(w)=\sum_{i \in \text{Des}(w)}i, \] 
and showed the remarkable result
that its distribution over all words on $S$ is equal to that of the inversion number 
over the same set.
Precisely, for the set  $W_S$ of all words on $S$, 
\begin{equation} 
\label{eq} \sum_{w \in W_S}{q^{\text{maj}(w)}} =\sum_{w \in
W_S}{q^{\text{inv}(w)}}.\end{equation} The proof of~\eqref{eq} given by MacMahon relied on combinatorial analysis, while an elegant bijection $\Phi:W_S \rightarrow W_S$ such that
$\m(w)=\text{inv}(\Phi(w))$ was constructed later by Foata~\cite{foata}.

A \textit{set partition} of $[n]$ is a family of nonempty
sets $B_1, \dots, B_k$ which are pairwise disjoint and whose union
is $[n]$. If all the blocks $B_i$ have at most two elements, the
set partition is called a \textit{matching}. If the size of every
block is exactly two, the  matching is called  a \textit{perfect matching}. A set
partition $\pi$ can be represented as a simple graph with vertex set
$[n]$ drawn on a horizontal line, where two vertices are joined by
an arc if and only if they represent consecutive elements in the same 
block of $\pi$ under the numerical order. 
Two statistics, the numbers of crossings and nestings, arise naturally from such
graphical representations, and have been the central topic in many recent research articles,
e.g.,\cite{CDDSY},~\cite{kz},~\cite{klazar},~\cite{P},~\cite{PY},~\cite{Riordan},~\cite{deSC},~\cite{touchard}, to list a few.
A \emph{crossing} (resp. \emph{nesting}) of a set partition is a pair of arcs
$(i_1,j_1)$, $(i_2,j_2)$ such that $i_1 < i_2 < j_1 < j_2$ (resp.
$i_1 < i_2 < j_2 < j_1$). A crossing  can be 
viewed as a generalization of an inversion on
permutations. Hence it is natural to ask what would
be the major index for matchings (set partitions) that extends MacMahon's 
equidistribution results \eqref{eq}. A definition of
such a statistic, denoted by pmaj, was proposed by Chen et al.
in~\cite{CGYY}. For completeness, we present the definition for
perfect matchings here. The general case for set partitions is very
similar. Label the arcs of a perfect matching $M$ of $[2n]$ by
$1,\dots,n$ in decreasing order with respect to their left
endpoints from left to right. Suppose $r_1, \dots, r_n$ are the
right endpoints of arcs from right to left. Then to each right
endpoint $r_i$ of an arc associate a word $w^{(i)}$ on $[n]$ where
the words $w^{(i)}$ are defined backwards recursively: let
$w^{(1)}=a$, where $a$ is the label of the arc that ends with
$r_1$. In general, after defining $w^{(i)}$, assume that the left
endpoints of the arcs labeled $a_1, \dots, a_t$ lie between $r_i$
and $r_{i+1}$. Then $w^{(i+1)}$ is obtained from $w^{(i)}$ by
deleting entries $a_1, \dots, a_t$ and adding the entry $b$ at the
very beginning, where $b$ is the label of the arc that ends with
$r_{i+1}$. Then
\[\text{pmaj}(M):=\sum_{i=1}^n \de(w^{(i)}).\] Chen et
al.~\cite{CGYY} computed the generating function for $pmaj$ and showed
that it is equally distributed with the number of crossings over all matchings with fixed left and right endpoints. A similar result holds for set partitions.

The aim of this paper is to extend the major index to certain 01-fillings of moon
polyominoes which include words and set partitions. In
Section~\ref{S:definitions} we introduce the necessary notations, and describe the 
 definition of the major index for 01-fillings of moon polyominoes. 
We explain in Section~\ref{examples} how the classical definition of major index
on words and set partitions can be obtained by considering 
moon polyominoes of special shapes. 
In Section~\ref{S:proof} we show that the maj statistic is equally distributed as ne, the
number of north-east chains, by computing the corresponding generating functions. 
 In the fillings of special shapes, ne corresponds to the
number of inversions for words and crossings for set partitions. Therefore, our main result,
Theorem~\ref{main}, leads to a generalization of~\eqref{eq} and the
analogous result for set partitions. 
In Section~\ref{S:foata},  we present a bijective proof for the equidistribution of 
maj and ne, which consists of three maps.  The first is from the fillings of left-aligned stack polyominoes to itself that sends $\text{maj}$ to $\text{ne}$, constructed in the spirit of Foata's second fundamental transformation.
The other two are maps that transform a moon polyomino to a left-aligned stack polyomino with the same set of columns, while preserving the statistics maj and ne, respectively.  Composing these three maps yields the desired bijection. 


\section{Definition of major index for fillings of moon polyominoes}\label{S:definitions}

A \textit{polyomino} is a finite subset of $\mathbb{Z}^2$, where
we represent every element of $\mathbb{Z}^2$ by a square cell. The
polyomino is \textit{convex} if its intersection with any column or row 
is connected. 
 It is \textit{intersection-free} if every
two columns are comparable, i.e., the row-coordinates of one
column form a subset of those of the other column. Equivalently, it is
intersection-free if every two rows are comparable. A \textit{moon
polyomino} is a convex intersection-free polyomino. If the rows
(resp. columns) of the moon polyomino are left-aligned (resp.
top-aligned), we will call it a \textit{left-aligned stack
polyomino} (resp. \textit{top-aligned stack polyomino}). A
\textit{Ferrers diagram} is  a left-aligned and top-aligned stack
polyomino. See Figure~\ref{moonexample} for an illustration.

\begin{figure}[ht]
\begin{center}
\includegraphics[width=10cm]{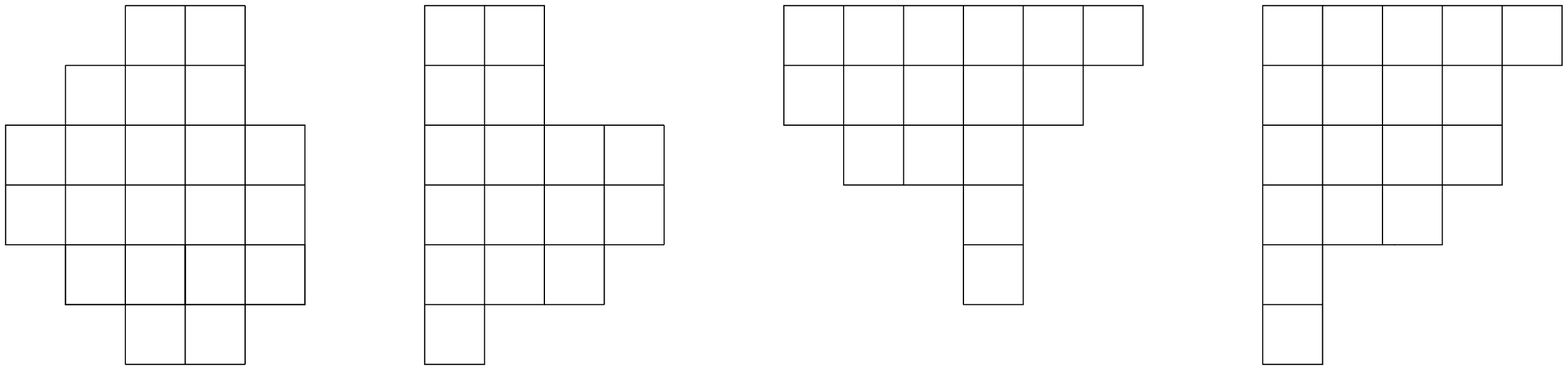}
\end{center}
\caption{A moon polyomino, a left-aligned and a top-aligned stack polyomino, and a Ferrers diagram.}
\label{moonexample}
\end{figure}

We are concerned with 01-fillings of moon
polyominoes with restricted row sums. That is, given a moon polyomino $\M$, we 
assign a $0$ or a $1$ to
each cell of $\M$ so that there is at most one $1$ in each row. 
Throughout this paper we will simply use
the term \textit{filling} to denote such 01-filling. 
We say a cell is empty if it is assigned a $0$, and it is a $1$-cell otherwise. 
Given a
moon polyomino $\M$ with $m$ columns, we label them $c_1,\dots,
c_m$ from left to right . Let $\textbf{s}=(s_1,\dots,s_m)$ be an
$m$-tuple of nonnegative integers and $A$ be a subset of rows of
$\M$. We denote by $\n(\M,\s;A)$ the set of fillings $M$ of
$\M$ such that the empty rows of $M$ are exactly those in $A$ and
the column $c_i$ has exactly $s_i$ many 1's, $1 \leq i \leq m$. A
\textit{north-east (NE) chain} in a filling $M$ of $\M$ is a set of
two $1$-cells such that one of them is strictly above and to the right of the
other and the smallest rectangle containing them is contained in
$\M$. The number of NE chains of $M$ will be denoted $\se(M)$. See Figure~\ref{moon}.

\begin{figure}[ht]
\begin{center}
\includegraphics[width=2.5cm]{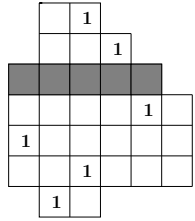}
\end{center}
\caption{Filling $M$ of a moon polyomino $\M$ for $A=\{3\}$ and
$\s=(1,1,2,1,1,0)$ with $\se(M)=5$.} \label{moon}
\end{figure}

The set $\n(\M,\s;A)$ was first studied by A.~Kasraoui \cite{kas}, who showed that the numbers of north-east and south-east chains are equally distributed. The generating functions for 
these statistics are also presented in \cite[Theorem 4.4]{kas}.

To define the major index for fillings of moon polyominoes, we
first state it for rectangular shapes, which is essentially the
classical definition of major index for words, (c.f. Section~\ref{examplerec}).

Let $R$ be a filling of a rectangle whose $n$ \textit{nonempty}
rows $r_1, \dots, r_n$ are numbered from top to bottom. We define
the \textit{descent statistic} for $R$ as
\[\text{des}(R)=|\{i\ |\ \text{ the $1$-cells in rows } r_i \text{
and } r_{i+1} \text{ form an NE chain}\}|.\] That is, descents of a
rectangular filling are NE chains in consecutive nonempty rows.

For each nonempty row $r_i$ of $R$, let $R(r_i)$ denote the
rectangle that contains the row $r_i$ and all rows in $R$ above
the row $r_i$. 

\begin{definition} \label{defrec}
The major index for  the rectangular f\/illing $R$ is defined to be
 \begin{equation} \label{majrec} \m(R)=\sum_{i=1}^{n}{\de(R(r_i))}.\end{equation}
\end{definition}
Clearly, the empty rows of $R$ do not play any role in $\m(R)$.
That is, if we delete them, the major index of the resulting
filling remains the same.

\begin{definition} \label{defmajal} Let $M$ be a filling of a moon polyomino $\M$. Let $\R_1, \dots, \R_r$ be the list of all the maximal rectangles contained in $\M$ ordered increasingly by height, and denote by $R_i$ the filling $M$ restricted on the rectangle $\R_i$. Then the major index of  $M$ is defined to be
 \begin{equation} \label{majmoonal} \m(M)=\sum_{i=1}^{r}{\m(R_i)} - \sum_{i=1}^{r-1}{\m(R_i \cap R_{i+1})}, \end{equation}
where $\m(R_i)$ and $\m(R_i \cap R_{i+1})$ are defined
by~\eqref{majrec}.
\end{definition}

In particular, $\m(M)=\m(R_1)$ when $M$ is a rectangular shape. 
It is also clear that $\m(M)$ is always nonnegative since
$\m(R_{i+1}) \geq \m(R_i \cap R_{i+1})$, $i=1,\dots,r-1$.
\begin{example} Consider the filling from Figure~\ref{moon}. It has four maximal rectangles and $\m(M)=(2+2+1+1)-(2+0+0)=4$.
\begin{figure}[ht]
\begin{center}
\includegraphics[width=11cm]{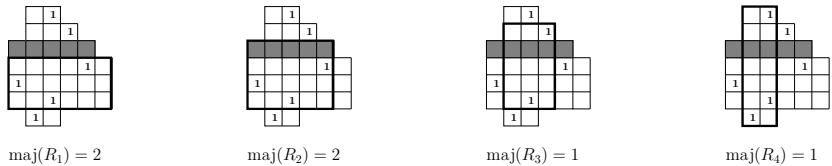}
\end{center}
\caption{Calculation of $\m(M)$ for a moon polyomino using Definition 2.2.}
\label{def1}
\end{figure}
\end{example}

The major index can be equivalently defined in a slightly more complicated way. However, this way is useful in proofs, especially when one uses induction on the number of columns of the filling, 
(c.f. Theorem 4.1). We state this equivalent definition next. First we need some notation. 
If the columns of $\M$
are $c_1, \dots, c_m$ from left to right, it is clear that the
sequence of their lengths is unimodal and there exists a unique
$k$ such that $|c_1| \leq \cdots \leq |c_{k-1}| < |c_k| \geq
|c_{k+1}| \geq \cdots \geq |c_m|$, where $|c_i|$ is the length of
the column $c_i$.  The left part of $\M$, denoted by $Left(\M)$, is
the set of columns $c_i$ with $ 1 \leq i \leq k-1$ and the right
part of $\M$, denoted by $Right(\M)$, is the set of columns $c_i$
with $ k \leq i \leq m$. Note that the columns of maximal length
in $\M$ belong to $Right(\M)$.

We order the columns of $\M$, $c_1, \dots, c_m$, by a total order $\prec$ as follows:
  $c_i \prec c_j$ if and only if
\begin{itemize}
 \item $\aci < \acj$ or
\item $\aci=\acj$, $\ci \in Left(\M)$ and $\cj \in Right(\M)$, or
\item $\aci=\acj$, $\ci, \cj \in Left(\M)$ and $\ci$ is to the left
of $\cj$, or \item $\aci=\acj$, $\ci, \cj \in Right(\M)$ and $\ci$
is to the right of $\cj$.
\end{itemize}
Similar ordering of rows was used in~\cite{kas}.

For every column $\ci \in Left(\M)$ define the rectangle $\M(\ci)$
to be the largest rectangle that contains $\ci$ as the leftmost
 column. For $ \ci \in
Right(\M)$, the rectangle $\M(\ci)$ is  taken to be the largest
rectangle that contains $\ci$ as the rightmost column and does not
contain any columns from $Left(\M)$ of same length as $c_i$.

\begin{definition*}[\textbf{Definition~\ref{defmajal}'.}] Let $c_{i_1} \prec c_{i_2} \prec \cdots \prec c_{i_m}$ be the ordering of the columns of $\M$ and let $M_j$ be the restriction of $M$ on the rectangle $\M(c_{i_j})$, $1 \leq j \leq m$. Then
\begin{equation} \label{majmoon}
\m(M)= \sum_{j=1}^{m}{\m(M_j)} - \sum_{j=1}^{m-1}{\m(M_j \cap M_{j+1})},
\end{equation}
where $\m(M_j)$ and $\m(M_j \cap M_{j+1})$ are defined
by~\eqref{majrec}.
\end{definition*} 

\begin{example} \label{majmoonex} Consider the f\/illing $M$ from Figure~\ref{moon}.
The order $\prec$ on the columns of $\M$ is $c_6 \prec c_1 \prec
c_5 \prec c_4 \prec c_3 \prec c_2$. So, $\M_1=\M(c_6)$,
$\M_2=\M(c_1)$, $\M_3=\M(c_5)$, $\M_4=\M(c_4)$, $\M_5=\M(c_3)$,
and $\M_6=\M(c_2)$, as illustrated in Figure~\ref{moonmaj1}.
By Definition 2.2', $\mathrm{maj}(M)=(2+2+1+1+1+0)-(2+1+0+0+0)=4$.

\begin{figure}[ht]
\begin{center}
\includegraphics[width=11cm]{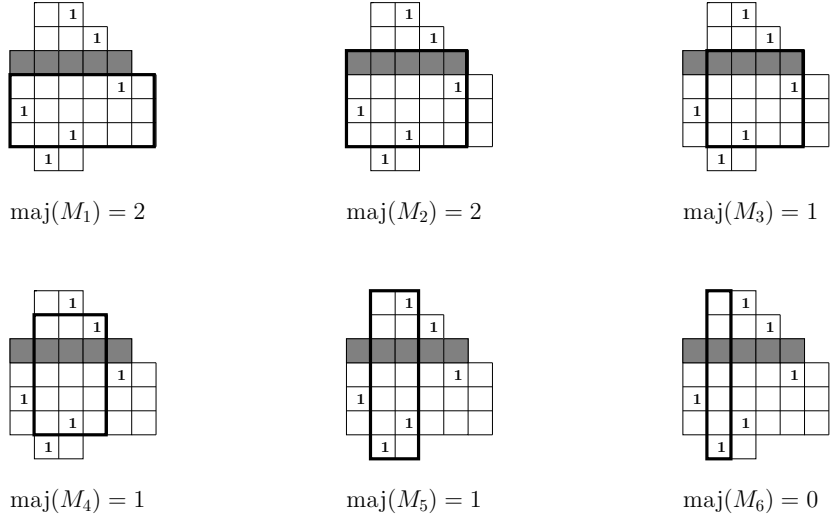}
\end{center}
\caption{Calculation of $\m(M)$ for a moon polyomino using Definition 2.2'.}
\label{moonmaj1}
\end{figure}
\end{example}

\begin{proposition} Definitions~\ref{defmajal} and 2.2' are equivalent.
\end{proposition}

\begin{proof} 
Some of the rectangular fillings $M_j$ in Definition 2.2' may contain $M_{j+1}$,
 in which case $\m(M_{j+1})-\m(M_j \cap M_{j+1})=0$ and formula~\eqref{majmoon} can be simplified. More precisely, there are uniquely determined
 indices $1=j_1 < j_2 < \cdots < j_r \leq m$ such that $M_1 =M_{j_1} \supseteq M_{2} \supseteq  \cdots \supseteq M_{j_2-1}
 \nsupseteq M_{j_2} \supseteq  \cdots \supseteq M_{j_r-1}
 \nsupseteq M_{j_r}  \supseteq \cdots \supseteq M_{m}$. 
That is, among the rectangle $M_i$'s the maximal ones are $M_{j_1}, M_{j_2},\dots, M_{j_r}$. 
Following the notation in  Definition 2.2, 
the filling $M_{j_k}$ is the filling $R_k$ of the $k$-th maximal rectangle contained in $\M$. Then, after cancellation of some terms, the right-hand side of ~\eqref{majmoon} becomes the right-hand side of~\eqref{majmoonal}.
\end{proof}



\section{$\m(M)$ for special shapes $\M$} \label{examples}

It is well-known that permutations and set-partitions can be
represented as fillings of Ferrers  diagrams (e.g.~\cite{demier}, ~\cite{Krattenthaler}). Here we will describe such presentations to show how to get the classical major index for words and set
partitions from Definition 2.2.

\subsection{When $\M$ is a rectangle: words and permutations} \label{examplerec}
For any rectangle, label the rows from top to bottom, and columns from left to right. 
Fillings of rectangles are in bijection with words. More
precisely, let $w=w_1 \dots w_n$ be a word with letters in  the
set $[m]$. The word $w$ can be represented as a filling $M$ of an
$n \times m$ rectangle $\M$  in which the cell in row $n-i+1$ and 
column $m-j+1$ is assigned the integer $1$ 
if $w_i=j$, and is empty otherwise. Conversely, each filling
of the rectangle $\M$ corresponds to a word $w$. 
It is readily to check that 
$\de(M)=\de(w)$ and $\m(M)=\m(w)$. This follows from the fact that
$\m(w)=\sum_{i=1}^{n}{\de(w_i w_{i+1} \dots w_n})$.

\subsection{When $\M$ is a Ferrers diagram: matchings and set
partitions}

As explained in~\cite{demier}, general fillings of Ferrers
diagrams correspond to multigraphs. Here we briefly describe the
correspondence when restricted to 01-fillings with row sum at
most 1. The Ferrers diagram is bounded by a vertical line from the
left,  a horizontal line from above, and a  path consisting of
east and north steps. Following this path starting from the bottom
left end, we  label the steps of the path by $1,2,\dots,n$.
This gives a labeling of the columns and
rows of the diagram. Draw $n$ vertices on a horizontal line and
draw an edge connecting vertices $i$ and $j$ if and only if there
is a 1 in the cell of the column labeled $i$ and the row labeled
$j$ (see Figure~\ref{ferrers1}). The resulting graph has no loops
or multiple edges and every vertex is a right endpoint of at most
one edge. The NE chains correspond to crossings of edges in the
graph. If the column sums are at most one, the graph is a matching
and the major index of the filling is equal to the major index of
the matching as defined in~\cite{CGYY}. As explained
in~\cite{CGYY}, the case of set partitions can be reduced to that
of matchings. Hence fillings of Ferrers diagrams include set
partitions as well. Restricting to the triangular Ferrers diagram, the 01-fillings 
considered in this paper also include \emph{linked
partitions}, a combinatorial structure that was studied in~\cite{CWY,kas}.

\begin{figure}[ht]
\begin{center}
\includegraphics[width=11cm]{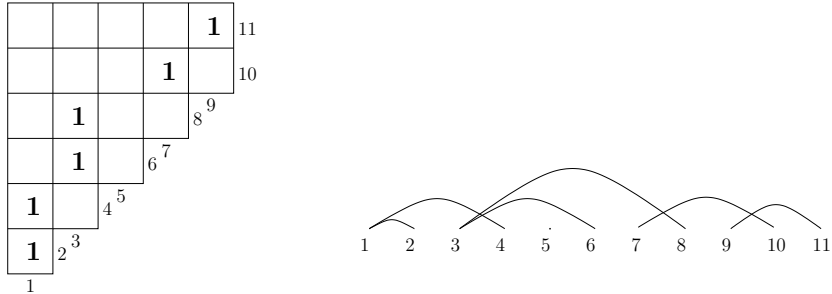}
\end{center}
\caption{A filling of a Ferrers diagram and the corresponding graph.}
\label{ferrers1}
\end{figure}

\subsection{When $\M$ is a top-aligned stack polyomino}

In the case when $\M$ is a top-aligned stack polyomino, $\m(M)$
has a simpler form in terms of the des statistic, just as in words and permutations. 
Explicitly, if the nonempty rows of $M$ are 
$r_1, \dots, r_n$, denote by $M(r_i)$ the filling $M$ restricted to the largest rectangle contained in $\M$ whose bottom row is $r_i$.

\begin{proposition} \label{stackprop} Let $M$ be a filling of a top-aligned stack polyomino and $M(r_i)$ be as defined above. Then
\begin{equation} \label{majstack} \m(M)=\sum_{i=1}^{n}{\de(M(r_i))}. \end{equation}
\end{proposition}

\begin{proof} 
Let $r_{k+1}, \dots, r_n$ be all the nonempty rows of $M$ that are of same length as the last row $r_n$. Denote by $M' = M \backslash \{r_{k+1}, \dots, r_n\}$ and by $R$  the filling of the largest rectangle of $\M$ containing $r_n$. We proceed by induction on the number $n$ of rows of $M$. The claim is trivial when $M$ has only one row or is a rectangle. Otherwise,  
\begin{align*}
\m(M)&=\m(M')+\m(R)-\m(M' \cap R)  &\text{ (Definition~\ref{defmajal})}\\
&=\sum_{i=1}^{k}\de(M(r_i)) +\sum_{i=k+1}^{n}\de(M(r_i)), &\text{ (Induction hypothesis and Definition~\ref{defrec})}
\end{align*}
which completes the proof.
\end{proof}


\begin{example} Consider the filling $M$ of a top-aligned stack polyomino in Figure~\ref{stackmaj} with $A=\{3\}$ and $\s=(0,1,2,1,1)$.
 $M$ has five nonempty rows and, by Proposition~\ref{stackprop}, $\m(M)=\sum_{i=1}^{5}\de(M(r_i))=0+1+0+1+1=3$.

\begin{figure}[ht]
\begin{center}
\includegraphics[width=14cm]{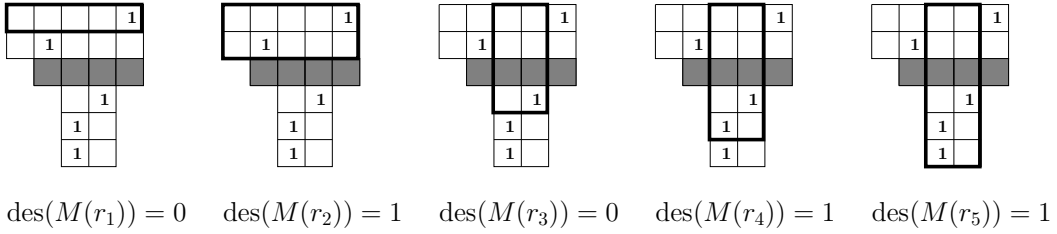}
\end{center}
\caption{Calculation of $\m(M)$ using Proposition~\ref{stackprop}.} \label{stackmaj}
\end{figure}
\end{example}

Proposition 3.1 will be used in the proof of Lemma 5.5.

\section{The generating function for the major index}\label{S:proof}

In this section we state and prove the main theorem which gives the generating function 
for maj over $\n(\M,\s;A)$, the set of all fillings of the moon polyomino $\M$
with column sums given by the integer sequence $\s$ and empty rows
given by $A$ . Let $|c_1|, \dots, |c_m|$ be the column lengths of
$M$ and let $a_i$ the number of rows in $A$ that intersect the
column $c_i$. Suppose that, under the ordering defined in
Section~\ref{S:definitions}, $c_{i_1} \prec c_{i_2} \prec \cdots
\prec c_{i_m}$. Then for $j=1, \dots, m$, define
\begin{equation} \label{hi}
h_{i_j}=|c_{i_j}|-a_{i_j} -(s_{i_1}+s_{i_2}+ \cdots +
s_{i_{j-1}})\end{equation}
The numbers $h_i$ have the following
meaning: if one fills in the columns of $M$ from smallest to
largest according to the order $\prec$, then $h_{i_j}$ is the number of
available cells in the $j$-th column  to be filled.

\begin{theorem} \label{main} For a moon polyomino $\M$ with $m$ columns,
\begin{equation} \label{maineq} \sum_{M \in \n(\M,\s;A)} q^{\m(M)}=\prod_{i=1}^{m}
\genfrac[]{0pt}{}{h_i}{s_i}_q \end{equation}
where $h_i$ is
defined by~\eqref{hi}.
\end{theorem}

We postpone the proof of Theorem~\ref{main} until the end of this
section.

\begin{example}
Suppose $\M$ is the first moon polyomino in
Figure~\ref{moonexample}, $A=\{5\}$, and $\s=(1,0,2,1,1)$. The
$\prec$ order on the columns of $\M$ is: $c_1 \prec c_5 \prec c_2
\prec c_4 \prec c_3$. The fifth row intersects all columns except
$c_1$, so, $a_1=0$ and $a_2 = a_3 = a_4 = a_5 = 1$. Therefore,
\begin{align*}h_1&= h_{i_1}= |c_1| - a_1 = 2 -0 =2\\
h_5&= h_{i_2}= |c_5| - a_5 -s_1= 3 -1 - 1=1\\
h_2&= h_{i_3}= |c_2| - a_2 -(s_1+s_5)= 4 -1 -(1+1) =1\\
h_4&= h_{i_4}= |c_4| - a_4 -(s_1+s_5+s_2)= 6 -1 -(1+1+0)=3\\
h_3&= h_{i_5}=|c_3| - a_3 -(s_1+s_5+s_2+s_4)= 6 -1 -(1+1+0+1) =2
\end{align*}
By Theorem~\ref{main},
$\sum_{M \in \n(\M,\s;A)} q^{\m(M)}=\prod_{i=1}^{m}
\genfrac[]{0pt}{}{h_i}{s_i}_q= \genfrac[]{0pt}{}{2}{1}_q \genfrac[]{0pt}{}{1}{0}_q\genfrac[]{0pt}{}{2}{2}_q\genfrac[]{0pt}{}{3}{1}_q \genfrac[]{0pt}{}{1}{1}_q =(1+q)(1+q+q^2)=1+2q+2q^2+q^3$. The fillings in $\n(\M,\s;A)$ are listed in Figure~\ref{theoremexample}.
\begin{figure}[ht]
\begin{center}
\includegraphics[width=9cm]{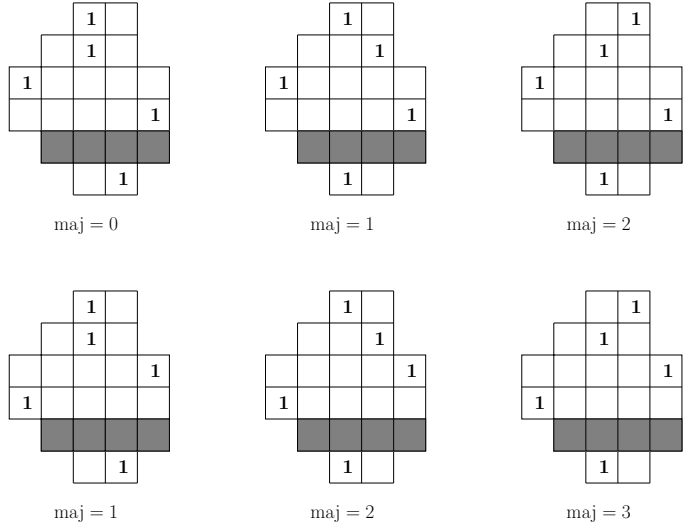}
\end{center}
\caption{Illustration of Theorem~\ref{main}.}
\label{theoremexample}
\end{figure}
\end{example}

\begin{corollary} Let $\sigma$ be a permutation of $[m]$ and let $\M$ be a moon polyomino with columns $c_1, \dots, c_m$. Suppose the shape $\N$ with columns $c_i'=c_{\sigma(i)}$ is also a moon polyomino and $\s'= (s_1', \dots, s_m')$ with $s_i'=s_{\sigma(i)}$. Then
\begin{equation}
\sum_{M \in \n(\M,\s;A)} q^{\m(M)}= \sum_{N \in \n(\N,\s';A)} q^{\m(N)}.
\end{equation}
That is, the generating function $\sum_{M \in \n(\M,\s;A)} q^{\m(M)}$ does
not depend on the order of the columns of $\M$.
\end{corollary}
\begin{proof}
The formula~\eqref{maineq} seems to depend on the order $\prec$ of the columns of the moon polyomino. However, note that $c_i \prec c_j$ if $|c_i| <|c_j|$; and 
columns of same length are consecutive in the order $\prec$.  Hence it suffices to compare the terms in the right-hand side of~\eqref{maineq} that come from columns of fixed length. If $c_{i_{j+1}} \prec \cdots \prec c_{i_{j+k}}$ are all the columns of $\M$ of fixed length, then 
\begin{equation}
h_{i_{j+r}}=h_{i_{j+1}}-(s_{i_{j+1}}+\cdots+ s_{i_{j+r-1}}), \; 2 \leq r \leq k.
\end{equation}
Thus the contribution of these columns to the right-hand side of~\eqref{maineq} is
\begin{equation} \prod_{r=1}^k \genfrac[]{0pt}{}{h_{i_{j+r}}}{s_{i_{j+r}}}_q = \genfrac[]{0pt}{}{h_{i_{j+1}}}{s_{i_{j+1}}, \dots, s_{i_{j+k}}, h_{i_{j+1}}- (s_{i_{j+1}}+ \cdots + s_{i_{j+k}}) }_q.
\end{equation}
The last q-multinomial coefficient is invariant under permutations of $s_{i_{j+1}}, \dots, s_{i_{j+k}}$, and the claim follows.

\end{proof}

The right-hand side of formula~\eqref{maineq} also appeared in 
the generating function of NE chains, as shown in~\cite{kas}:

\begin{theorem}[Kasraoui]\label{kastheorem}
\[\sum_{M \in \n(\M,\s;A)} q^{\se(M)}=\prod_{i=1}^{m} \genfrac[]{0pt}{}{h_i}{s_i}_q. \]
\end{theorem}

\noindent Combining Theorems~\ref{main} and~\ref{kastheorem}, we have

\begin{theorem}\label{equal}
\[\sum_{M \in \n(\M,\s;A)} q^{\m(M)}=\sum_{M \in \n(\M,\s;A)} q^{\se(M)}. \]
\end{theorem}
\noindent That is, the maj statistic  has the same distribution as the ne statistic over the set
$\n(\M, \s, A)$.  Since $\se(M)$ is the natural 
analogue of inv for words and permutations, Theorem \ref{equal} generalizes
MacMahon's equidistribution result.

We shall prove a lemma about the major index for
words and then use it to prove Theorem~\ref{main}.

\begin{lemma} \label{perm}
 Let $w=w_1\dots w_k$ be a word such that $w_i < n$ for all $i$. Consider the set $S(w)$ of
 all the words $w'$ that can be obtained by inserting $m$ many $n$'s between the
 letters of $w$. Then the difference $\m(w')-\m(w)$ ranges over the
 multiset $\{i_1+i_2+\cdots +i_m\  |\  0 \leq i_1 \leq i_2 \leq \cdots \leq i_m \leq
 k\}$. It follows that
\begin{equation}\label{lemma}
\sum_{w' \in S(w)}q^{\m(w')}=q^{\m(w)}\sum_{0 \leq i_1 \leq i_2
\leq \cdots \leq i_m \leq k}q^{i_1+i_2+\cdots +i_m}=q^{\m(w)}
\genfrac[]{0pt}{}{k+m}{m}_q.\end{equation} The same statement
holds if $w=w_1\dots w_k$ is a word such that $w_i >1$ for all $i$
and $S(w)$ is the set of all the words $w'$ obtained by inserting
$m$ many 1's between the letters of $w$.
\end{lemma}

\begin{proof} 
We give an elementary proof for the case of inserting $n$'s.
The case of inserting 1's is dealt with similarly. For a word
$w=w_1 \dots w_k$ we define the \textit{descent sequence}
$\ds(w)=a_1 \dots a_{k+1}$ by letting $a_i=\m(w^{(i)})-\m(w)$, where
$w^{(i)}$ is obtained by inserting one $n$ in the $i$-th gap of
$w$, i.e., $w^{(i)}=w_1 \dots w_{i-1} n w_{i} \dots w_k$.  Note
that
\[\m(w^{(i)})=\m(w)+\begin{cases} 1+ \de(w_i \dots w_k), &\text{ if }
w_{i-1} >w_i \text{ or } i=1\\
i+\de(w_i \dots w_k), &\text{ if } w_{i-1} \leq w_i\\
0, &\text{ if } i=k+1.
\end{cases}
\]
Hence, if $t=\de(w)$ then $\ds(w)$ is a shuffle of the sequences
$t+1,t,\dots, 0$ and $t+2, t+3, \dots, k$ which begins with $t+1$
and ends with $0$. Moreover, every such sequence is $\ds(w)$ for
some word $w$. It is not hard to see that the multiset
$\{\m(w')-\m(w) \ |\  w' \in S(w)\}$ depends only on $\ds(w)$ and not
on the letters of $w$.

First we consider the case when $\ds(w)=k (k-1) \dots 0$, i.e.,
$\text{Des}(w)=[k-1]$. Inserting $b$ $n$'s between $w_{i-1}$ and
$w_i$ increases the major index by $b(k-i+1)$. Hence, if one
inserts $b_i$ many $n$'s between $w_{i-1}$ and $w_i$ for $1 \leq i
\leq k+1$,  the major index is increased by
\[\sum_{i=1}^{k+1} b_i(k-i+1)=\underbrace{k+\cdots +k}_{b_1}+\underbrace{(k-1)+\cdots+(k-1)}_{b_2}+\cdots+\underbrace{0+\cdots+0}_{b_{k+1}}.\]
Letting $(b_1,\dots, b_{k+1})$ range over all $(k+1)$-tuples with
$\sum_{i=1}^{k+1}b_i=m$, the increased amount ranges over the
multiset $\{i_1+i_2+\cdots +i_m\  |\ 0 \leq i_1 \leq i_2 \leq \cdots
\leq i_m \leq k\}$.

To deal with the other possible  descent sequences, we note that
any descent sequence  $\ds(w)$ can be obtained by applying a
series of adjacent transpositions to $k (k-1) \dots 0$, where we
only transpose adjacent  elements $ i, j$ which correspond to ascent and descent positions, respectively. 
Suppose
$\ds(w)=a_1 \dots a_{i-1} a_{i+1} a_{i} a_{i+2} \dots a_{k+1}$ is
obtained by transposing the $i$-th and $(i+1)$-st element in
$\ds(v)=a_1 \dots a_{i-1} a_{i} a_{i+1} a_{i+2} \dots a_{k+1}$,
where $a_i$ and $a_{i+1}$ correspond to an ascent and a descent in $v$, respectively. 
Consequently, $w_{i-1} > w_i \leq w_{i+1}$ and  
$v_{i-1} \leq v_i > v_{i+1}$, while $w_j \leq w_{j+1}$ if and only if
$v_j \leq v_{j+1}$ for $j \neq i-1, i$. 
It suffices to show that the multisets
$\{\m(w')-\m(w) \ |\  w' \in S(w)\}$ and $\{\m(v')-\m(v)\  |\  v' \in
S(v)\}$ are equal. Suppose $w' \in S(w)$ is obtained by inserting
$b_j$ many $n$'s in the $j$-th gap of $w$, $1 \leq j \leq k+1$.
Let $v' \in S(v)$ be the word obtained by inserting $c_j$ $n$'s in
the $j$-th gap of $v$ where $c_j$ are defined as follows.\\
If $b_{i+1}=0$ then
\begin{equation} 
c_j=\begin{cases} 0, &\text{ if } j=i\\
b_i , &\text{ if } j=i+1 \\
b_j, &\text { if } j \neq i, i+1.
\end{cases}
\end{equation}
If $b_{i+1}>0$ then
\begin{equation}
c_j=\begin{cases} b_i+1, &\text{ if } j=i\\
b_{i+1}-1 , &\text{ if } j=i+1 \\
b_j, &\text { if } j \neq i, i+1.
\end{cases}
\end{equation}
This defines a map between sequences $(b_j)$ with $b_{i+1}=0$
(resp. $b_{i+1}>0$), and $(c_j)$ with $c_i=0$ (resp. $c_i>0$),
which is a bijection. The $n$'s inserted in the $j$-th gap of $w$ for $j \neq i, i+1$ contribute to $\m(w')-\m(w)$ the same amount as do the $n$'s inserted in the same gaps in $v$ to $\m(v')-\m(v)$. The major index contributed to $\m(w')$  by the segment 
\[
 w_{i-1}\underbrace{n\dots n}_{b_i} w_i \underbrace{n \dots n}_{b_{i+1}} 
\]
is the sum of $i-1$ (which equals the contribution of $w_{i-1}w_i$ to $\text{maj}(w)$) and 
\begin{eqnarray*}
\begin{cases} (b_1+ \cdots + b_i)+ (b_1 + \cdots +b_{i+1}+i), &\text{if $b_{i+1} > 0$}\\
(b_1+ \cdots + b_i), &\text{if $b_{i+1} = 0$}.
\end{cases}
\end{eqnarray*}
Similarly, the major index contributed to $\m(v')$ by the segment 
\[
v_{i-1}\underbrace{n\dots n}_{c_i} v_i \underbrace{n \dots n}_{c_{i+1}}
\]
is the sum of $i$ (which equals the contribution of $v_{i-1}v_i$ to $\m(v)$) and
\begin{eqnarray*}
\begin{cases} 
(c_1+ \cdots + c_i +i-1)+ (c_1 + \cdots +c_{i+1}), &\text{if $c_i > 0$}\\
(c_1+ \cdots + c_{i+1}), &\text{if $c_i = 0$}.
\end{cases}
\end{eqnarray*}
Now, one readily checks that 
$\m(w')-\m(w)=\m(v')-\m(v)$. This completes the proof.

\end{proof}

\begin{remark}\label{gessel}
Lemma \ref{perm} can be proved in many ways, for example, in \cite[Theorem 1.2]{foata-han} Foata and Han constructed a transformation on the shuffle class Sh$(0^m v)$, 
where $v$ is a nonempty word whose letters are positive integers. 
Alternatively, one can use the theory of 
$P$-partitions \cite{gessel}, 
which leads to a stronger result, namely, 
Equation~\eqref{lemma}  holds if $w=w_1\dots w_k$ with $w_i \neq
n_0$, $i=1,\dots, k$ and $S(w)$ is the set of all words $w'$
obtained by inserting $m$ copies of $n_0$ between the letters of
$w$. To show this, it is enough to consider the case when $w$ is a
permutation of length $k$ on $k$ distinct numbers between 1 and
$k+m$ such that the integers $N_1< N_2< \dots< N_{m}$ are missing
from it and $w'$ varies over the set $S(w)$ of all permutations
obtained by inserting $N_1, N_2, \dots, N_m$ in $w$ from left to
right. We outline the proof in this general case assuming that the
reader is familiar with the theory of $P$-partitions. We adopt the
notation and definitions from Section 4.5 in~\cite{stanley}, which
is also a good reference for basic results on $P$-partitions.

Let $P_1$ be the antichain with $k$ elements labeled with the
letters of $w$. Let $P_2$ be a disjoint union of $P_1$ and a chain
of $m$ elements labeled $N_1 < N_2 < \cdots < N_m$. We denote by
$F_w(q_1,\dots, q_k)$ the generating function of all
$w$-compatible functions $f: [k] \rightarrow \mathbb{N}$ (see~\cite{stanley}
for the definitions). Lemma 4.5.2 in~\cite{stanley} asserts that 
\begin{equation} \label{A}
 F_w (q, \dots, q)= \frac{q^{\m(w)}}{(1-q) \cdots (1-q^k)}.
\end{equation}
All $P_2$-partitions compatible with some $w' \in S(w)$ are
obtained by combining a $P_1$-partition compatible with $w$ and a
integer partition with at most $m$ parts. Therefore,
\begin{equation} \label{B}
 \sum_{w' \in S(w)} F_{w'}(q,\dots,q) = \frac{F_w (q, \dots, q)}{(1-q) \cdots(1-q^m)}.
\end{equation}
Applying Lemma 4.5.2 in~\cite{stanley} to all $w' \in S(w)$, we
have
\begin{equation} \label{C}
 \sum_{w' \in S(w)} F_{w'}(q,\dots,q) = \sum_{w' \in S(w)} \frac{q^{\m(w')}}{(1-q) \cdots (1-q^{k+m})}.
\end{equation}
Combining equations~\eqref{A}, \eqref{B}, and \eqref{C} gives~\eqref{lemma}.

\end{remark}

\begin{proof}[\textbf{Proof of Theorem~\ref{main}}]

Let $c$ be the smallest  column of $\M$ in the order $\prec$ and
$\M'=\M \backslash c$. Note that $c$ is the leftmost or the
rightmost column of $\M$. Let $M$ be a filling of $\M$ with $s$
nonempty cells in $c$ and let $M'$ be its restriction on $\M'$.
From Definition 2.2', we derive that 
\[\m(M)=\m(M')+\m(R_c) - \m(R_c \cap M'),\] where $R_c$ is the
filling of the rectangle $\M(c)$ determined by the 
column $c$.
$R_c$ is obtained by adding a column with $s$ many $1$-cells
to the rectangular filling $R_c \cap M'$. Using the bijection
between fillings of rectangles and words described in
Section~\ref{examplerec}, one sees that if $c$ is the leftmost
(resp. rightmost) column in $M$, this corresponds to inserting
$s$ maximal (resp. minimal) elements in the word determined by
the filling $R_c \cap M'$. When the column $c$ varies over all
possible $\binom{h}{s}$ fillings, by Lemma~\ref{lemma}, the value
$\m(R_c) - \m(R_c \cap M')$ varies over the multiset
$\{i_1+i_2+\cdots +i_s\  | \ 0 \leq i_1 \leq i_2 \leq \cdots \leq i_s
\leq (h-s)\}$ with generating function $\sum_{0 \leq i_1 \leq
\cdots \leq i_s \leq (h-s)} q^{i_1+\cdots +i_s}=
\genfrac[]{0pt}{}{h}{s}_q$, where $h$ is the value $h_1$ (resp.
$h_m$) if $c$ is the leftmost (resp. the rightmost) column of
$M$, as defined in~\eqref{hi}. Therefore,
\[\sum_{M \in \n(\M,\s;A)} q^{\m(M)}=\genfrac[]{0pt}{}{h}{s}_q \sum_{M' \in \n(\M', \s';A)}
q^{\m(M')}\] where $\s'$ is obtained from $\s$ by removing the first (resp.
last) component if $c$ is the leftmost (resp. rightmost) column
of $M$. Equation~\eqref{maineq} now follows by induction on the
number of columns of $\M$.
\end{proof}

\section{The Foata-type bijection for moon polyominoes}\label{S:foata}

The objective of this section is to give a bijective  proof to Theorem 4.5. 
For the set $W_S$ of all words of a multiset $S$,
the equidistribution of maj and inv  was first proved by 
MacMahon by combinatorial analysis. 
This raised the question of 
constructing a canonical bijection $\Phi: W_S \rightarrow W_S$ such that 
$\text{maj}(w)= \text{inv}(\Phi(w))$. Foata answered the question by 
constructing an elegant map $\Phi$ \cite{foata}, which is referred to as the 
\emph{second fundamental transformation} \cite{foata97}.  
We begin this section by reviewing Foata's map $\Phi: W_S \rightarrow W_S$.

 Let $w=w_1w_2\cdots w_n$ be a word on $\mathbb{N}$ and let $a$ be an integer. 
 If $w_n \leq a$, the $a$-factorization of $w$ is
$w=v_1b_1\cdots v_pb_p$, where each $b_i$ is a letter less than or equal to $a$,
and each $v_i$ is a word (possibly empty), all of whose letters are
greater than $a$. Similarly, if $w_n > a$, the $a$-factorization of
$w$ is $w=v_1b_1\cdots v_pb_p$, where each $b_i$ is a letter greater
than  $a$, and each  $v_i$ is a word (possibly empty), all of whose
letters are less than or equal to $a$. In each case one defines
$$
\gamma_a(w)=b_1v_1\cdots b_pv_p.
$$
With the above notation, let $a=w_n$ and let $w'=w_1\cdots w_{n-1}$.
The second fundamental transformation $\Phi$ is defined recursively by
$\Phi(w)=w$ if $w$ has length $1$, and
$$
\Phi(w)=\gamma_a(\Phi(w'))a,
$$
if $w$ has length $n >1$. The map $\Phi$ has the property that
it preserves the last letter of the word, and 
$\text{inv} (\Phi(w)) =\text{maj}(w)$. 

Foata's map $\Phi$  is constructed recursively with certain ``local operations'' 
to eliminate the difference caused by adding the last letter in the words. Inspired 
by this idea,  we construct a Foata-type bijection $\phi:
\n(\M,\s;A) \rightarrow \n(\M,\s;A)$ with the property
$\m(M)=\se(\phi(M))$. The map $\phi$ can be defined directly for
left-aligned stack polyominoes. But we describe it first for Ferrers
diagrams in Section 5.1, because this case contains all the
essential steps and is easy to understand. This map is a revision of
the Foata-type bijection presented in the preprint of~\cite{CGYY}
for set partitions, which correspond to fillings with row and
column sums at most 1. Then in Section 5.2 we extend the construction to 
 left-aligned stack polyominoes and prove that
$\m(M)=\se(\phi(M))$. In Section 5.3, we construct two bijections,
$f$ and $g$, that transform a filling of a moon polyomino to a
filling of a left-aligned stack polyomino and preserve the statistics maj and
ne, respectively. Composing these maps  with $\phi$
defined on left-aligned stack polyominoes yields a bijection on
$\n(\M, \s; A)$ that sends maj to ne.

\subsection{The bijection $\phi$ for Ferrers diagrams}
\label{ferrers}

The empty rows in the filling $M \in \n(\M,\s;A)$ do not play any role in the definitions of $\m(M)$ or $\se(M)$. 
Therefore, in what follows we assume that $A=\emptyset$ and describe $\phi$ for fillings without empty rows. For $A \neq \emptyset$, one can first delete the empty rows of the filling, apply the map $\phi$, and reinsert the empty rows back.

Let $\F$ be a Ferrers diagram and $F$ a filling of $\F$. The
bijection $\phi$ is defined inductively on the number of rows of
$F$, which preserves the first row of $F$. 
If $F$ has only one row, then $\phi(F)=F$. Otherwise, 
we denote by $F_1$ the filling obtained by deleting the top
row $r$ of $F$. Let $F_1'=\phi(F_1)$ and let $F_2$ be the
filling obtained by performing the algorithm $\gamma_r$ described
below. Then $F'=\phi(F)$ is obtained
from $F_2$ by adding the top row $r$.

\subsection*{\textbf{\underline{The algorithm $\gamma_r$}}}

If $C$ is the $1$-cell of $r$, then denote by $\R$ the set of
all rows of $F_1'$ that intersect the column of $C$. The 
$1$-cells in $\R$ that are strictly to the left of $C$ are called left
and the $1$-cells in $\R$ that are weakly to the right of
$C$ are called right. The cell $C$ is neither left nor right.

Let $CLC$ (critical left cell) be the topmost left $1$-cell, and $CRC$
(critical right cell) be the leftmost right $1$-cell that is above
$CLC$. If there is more than one such cell the $CRC$ is defined to be the lowest one. Denote by $\R_1$ the set of all rows weakly below
the row of $CRC$ that intersect the column of $CRC$ and $\R_2$ is
the set of all rows in $\R$ that do not intersect the column of
$CRC$. If $CLC$ does not exist then both $\R_1$ and $\R_2$ are empty. If $CLC$ exists but $CRC$ does not, then $\R_1$ is empty and $\R_2$ contains all the rows in $\R$ weakly below $CLC$.

\begin{definition} \label{expressions}
Let  $\cone$ and $\ctwo$ be two $1$-cells with
coordinates $(i_1,j_1)$ and $(i_2,j_2)$, respectively. We \textit{swap the cells $\cone$ and
$\ctwo$} by deleting the 1's from these two cells and write 1's
in the cells with coordinates $(i_1,j_2)$ and $(i_2,j_1)$.
\end{definition}

For a cell $C$, denote by $col(C)$  the column of $C$, and $|col(C)|$  the length 
of $col(C)$.

\subsection*{\textbf{Algorithm $\gamma_r^1$ on $\R_1$}} 
\textsf{Let $\pone$ and $\ptwo$ be two pointers.}
\begin{enumerate}
\item [\textsf{\textbf{(A)}}] \textsf{Set $\pone$ on the highest row of $\R_1$ and
$\ptwo$ on the next row in $\R_1$ below $\pone$.}
\item [\textsf{\textbf{(B)}}] \textsf{If $\ptwo$ is null, then go to (D).
Otherwise, the pointers $\pone$ and $\ptwo$ point at $1$-cells $\cone$ and $\ctwo$, respectively.}
\textsf{\begin{enumerate}
\item [\textbf{(B1)}] If $\ctwo$ is a left cell then swap the
cells $\cone$ and $\ctwo$  and move $\pone$ to the row of $\ptwo$.
\item [\textbf{(B2)}] If $\ctwo$ is a right cell, then
\begin{enumerate}
\item [\textbf{(B2.1)}] If $|col(\cone)|=|col(\ctwo)|$ then
move $\pone$ to the row of $\ptwo$.
\item [\textbf{(B2.2)}] If $\cone$ is to the left of $\ctwo$ and
$|col(\cone)| > |col(\ctwo)|$ then do nothing.
\item [\textbf{(B2.3)}] If $\cone$ is to the right of $\ctwo$ and
$|col(\cone)| < |col(\ctwo)$,
 then find the lowest left $1$-cell $L$ that is above $\cone$.
Suppose that the row-column coordinates of the cells $L$, $\cone$,
and $\ctwo$ are $(i_1,j_1)$, $(i_2,j_2)$, and $(i_3,j_3)$,
respectively. Delete the 1's from these three cells and write them
in the cells with coordinates $(i_1,j_3)$, $(i_2,j_1)$, and
$(i_3,j_2)$. Move $\pone$ to the row of $\ptwo$.$^\dagger$
\end{enumerate}
\end{enumerate}}
\item [\textsf{\textbf{(C)}}] \textsf{Move $\ptwo$ to the next  row in $\R_1$.
Go to (B).
\item [\textbf{(D)}] Stop.}
\end{enumerate} 
See Figure~\ref{alg1} for illustration of the steps.

\textbf{$^\dagger$Note:}
This step is well-defined since it cannot occur before $\ptwo$ reaches $CLC$ (that would contradict the definition of $CRC$) and, after that, $\pone$ is always below $CLC$. Therefore, the $1$-cell $L$ always exists. The fact that the square $(i_3, j_2)$ belongs to $\mathcal{F}$ follows from the definition of $\R_1$ and Lemma 5.2 (c).

\begin{figure}[ht]
\begin{center}
\includegraphics[width=15cm]{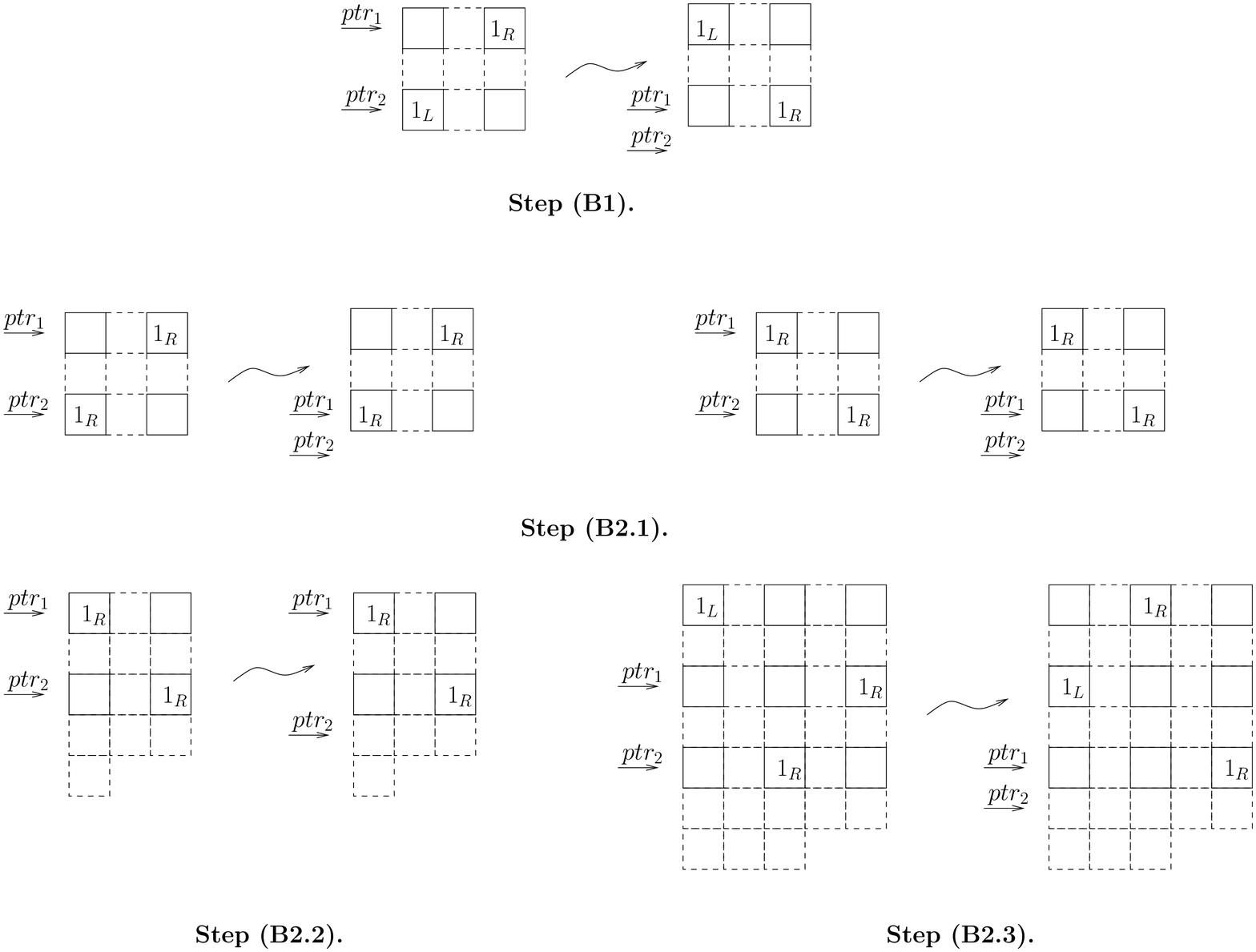}
\end{center}
\caption{The algorithm $\gamma_r^1$ on $\R_1$.}
\label{alg1}
\end{figure}

When the algorithm $\gamma_r^1$ stops, we continue processing the rows of $\R_2$ by algorithm $\gamma_r^2$.

\subsection* {\textbf{The algorithm $\gamma_r^2$ on $\R_2$}}

\begin{enumerate}
\item [\textsf{\textbf{(A$'$)}}]\textsf{ Set $\pone$ on the highest row of $\R_2$
and $\ptwo$ on the next row in $\R_2$ below $\pone$.}
\item [\textsf{\textbf{(B$'$)}}] (\textit{\textbf{Borrowing}})\textsf{ If $\pone$
points to a right $1$-cell then find the lowest left $1$-cell above it and
swap them. Now $\pone$ points to a left $1$-cell.}
\item [\textsf{\textbf{(C$'$)}}]  \textsf{If $\ptwo$ is null, then go to (E$'$).
Otherwise, the pointers $\pone$ and $\ptwo$ point at 
$1$-cells $\cone$ and $\ctwo$ respectively.}
\textsf{
\begin{enumerate}
\item [\textbf{(C$'$1)}] If $\ctwo$ is a right cell then swap
$\cone$ and $\ctwo$.
\item [\textbf{(C$'$2)}] If $\ctwo$ is a left cell then do
nothing.
\end{enumerate}}
\item [\textsf{\textbf{(D$'$)}}] \textsf{Move $\pone$ to the row of $\ptwo$ and
$\ptwo$ to the next row in $\R_2$ below. Go to (C$'$).
\item [\textbf{(E$'$)}] Stop. }
\end{enumerate}
See Figure~\ref{alg2} for illustration of the steps in $\gamma_r^2$.

\begin{figure}[ht]
\begin{center}
\includegraphics[width=15cm]{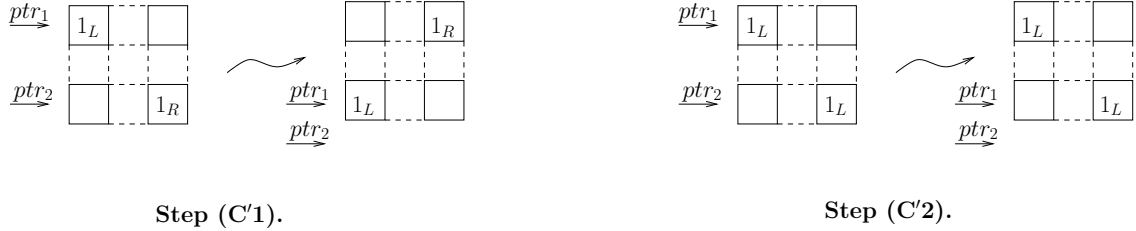}
\end{center}
\caption{Algorithm $\gamma_r^2$ on $\R_2$.}
\label{alg2}
\end{figure}

Let us note the following easy but useful properties of the algorithm $\gamma_r$.
\begin{enumerate}
\item The pointers $\pone$ and $\ptwo$ process the rows of $\R_1$
and $\R_2$ from top to bottom. However, while $\ptwo$ always moves
from one row to the next one below it, $\pone$ sometimes stays on
the same row (cf. step (B2.2)) and sometimes ``jumps'' several rows below (cf. step (B1)).

\item Pointer $\pone$ always points to a right $1$-cell in $\gamma_r^1$
and to a left $1$-cell in $\gamma_r^2$.

\end{enumerate}

Further we  show some nontrivial properties which will be used to show that the map $\phi$ has the desired property $\m(F)=\se(\phi(F))$.

\begin{lemma} \label{prop} 
Suppose in $\gamma_r^1$ $\pone$ is pointing to $\cone$ and $\ptwo$ is pointing to $\ctwo$. Then
\begin{itemize}
\item[(a)]  In  the rows between $\cone$ and $\ctwo$ there are no 1's weakly to the left of $\cone$.

\item[(b)]  If $L$ is the lowest left $1$-cell above $\cone$, then in the rows between $L$ and $\cone$ there are no 1's weakly to the left of $\cone$.

\item[(c)]  The $1$-cell $\cone$ is in a column with same length as the column of the critical right cell $CRC$. Moreover, when $\gamma_r^1$ stops $\pone$ points to the last row in $\R_1$.

\end{itemize}
\end{lemma} 

\begin{proof} The first two parts are proved by induction on the number of steps performed in the algorithm.
\begin{itemize}
\item [(a)] The only step which leaves the gap between $\pone$ and $\ptwo$ nonempty is (B2.2). But  after this step, the number of 1's in the rows between $\pone$ and $\ptwo$ weakly to the left of $\cone$ does not change.

\item[(b)] The case when the last step is (B2.2) follows from the induction hypothesis. If the last step is (B1) or (B2.3) the claim follows from part (a). The case when the last step was (B2.1) and the cells $\cone$ and $\ctwo$ formed an NE chain the claim also follows from part (a) and the induction hypothesis. If in (B2.1), the cells $\cone$ and $\ctwo$ do not form an NE chain,  in addition to the induction hypothesis, one uses the fact that there are no 1's in the rows between these two cells which are weakly to the left of $\ctwo$.

\item[(c)] The first part follows immediately from the definition of the steps. If the second part is not true, then the last step must be (B2.2).  and the column of $\cone$ is longer than the column of $\ctwo$. Since this is the last step of $\gamma_r^1$, it follows that the column of $\cone$ intersects the top row of $\R_2$. However, the column of $CRC$ does not intersect the top row of $\R_2$ by definition of $\R_2$. This contradicts the fact that $\cone$ and $CRC$ are in columns of same length.
\end{itemize}
\end{proof}

\begin{proposition} \label{alg} After algorithm $\gamma_r^1$ the number of NE chains decreases by the total number of  left $1$-cells in $\R_1$, whereas after algorithm $\gamma_r^2$ the number of NE chains increases by the total number of  right $1$-cells in $\R_2$.
\end{proposition}

\begin{proof} To prove the first part note that pointer $\ptwo$ points to each left $1$-cell in $\R_1$ exactly once in the algorithm $\gamma_r^1$. When that happens, step (B1) is performed during which, by Lemma~\ref{prop}(a), the number of NE chains decreases by one. In the other steps, the number of NE chains remains unchanged. This is trivial for steps (B2.1) and (B2.2), whereas for step (B2.3) it follows from Lemma~\ref{prop}~(a) and~(b).

For the second part, note that pointer $\ptwo$ points to each right $1$-cell in $\R_2$ exactly once in the algorithm $\gamma_r^2$ with possible exception being the top $1$-cell in $\R_2$ to which it never points. In those steps the number of NE chains is increased by one, while otherwise it remains the same. When the top $1$-cell is right, then
Borrowing occurs. It follows from parts~(b) and ~(c) of Lemma~\ref{prop} that the number of NE chains is also increased by one.
\end{proof}

\begin{figure}[ht]
\begin{center}
\includegraphics[width=15cm]{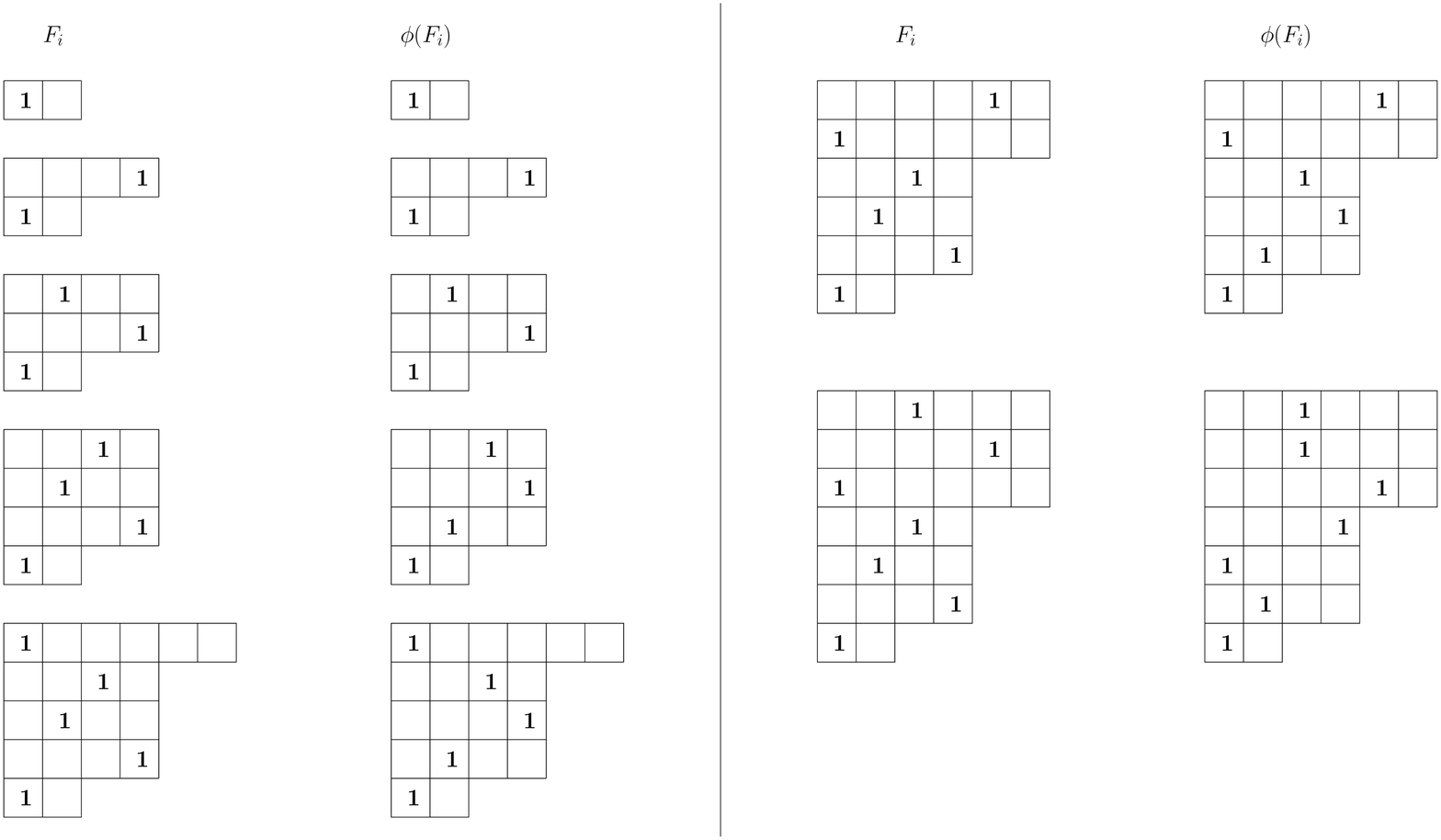}
\end{center}
\caption{Example of the map $\phi$ applied inductively on $F=F_7$. The fillings $F_i$ are restrictions of $F$ on the last $i$ rows.}
\label{bijex}
\end{figure}

\begin{theorem}\label{ferrer}
The map $\phi: \n(\F,\s;A) \rightarrow \n(\F,\s;A)$ is a bijection.
\end{theorem}
The proof is quite technical and is given in the appendix. Here we only present an example of the bijection $\phi$ in Figure~\ref{bijex} and explain how  it includes Foata's 
second fundamental transformation. If the moon polyomino is rectangular, then the algorithm $\gamma_r$ has a simpler description because some of the cases can never arise. Namely, if the cell in the second row forms an NE chain with the top cell (a descent) then $\R_1 = \emptyset$, and only the steps (C'1) and (C'2) are performed. Otherwise, $\R_2 = \emptyset$ and, since all columns are of equal height, only the steps (B1) and (B2.1) are performed. In both these cases, via the correspondence between words and rectangular fillings (c.f. Section~\ref{examplerec}) the algorithms $\gamma_r^2$ and $\gamma_r^1$ for rectangles are equivalent to Foata's transformation for words~\cite{foata}.

\subsection{The bijection $\phi$ for left-aligned stack
polyominoes}\label{stack}

Next we extend the bijection $\phi$ to fillings of 
left-aligned stack polyominoes which sends maj to ne. 
Let $\M$ be a left-aligned stack polyomino. 
 As in the case of Ferrers diagrams, we can assume that $A = \emptyset$ and only consider fillings
 without empty rows.  Suppose $M$ is a filling of  $\M$ with top row $r$. 
The map, which is again denoted by $\phi$, 
 is defined inductively on the number of rows of $M$. 
If $M$ has only one row, then $\phi(M)=M$. Otherwise, 
let $\F$ be the maximal Ferrers
diagram in $\M$ that contains the top row and $\F_1$ be $\F$ without the top row. 
Let $F=M \cap \F$ and $F_1 =M \cap \F_1$. To obtain $\phi(M)$ we perform the following steps. (See Figure~\ref{stack1}). 
\begin{enumerate}
 \item Delete the top row of $M$ and get $M_1 = M \backslash r$.
 \item  Apply $\phi$ to $M_1$ and get $M_1'=\phi(M_1)$.
\item Apply the algorithm $\gamma_r$ to the filling $F_1'= M_1' \cap \F_1$ and leave the cells in $M_1'$ outside of $F_1'$ unchanged. Denote the resulting filling by $M_2$.
\item $\phi(M)$ is obtained by adding row $r$ back to $M_2$.
\end{enumerate}

\begin{figure}[ht]
\begin{center}
\includegraphics[width=15cm]{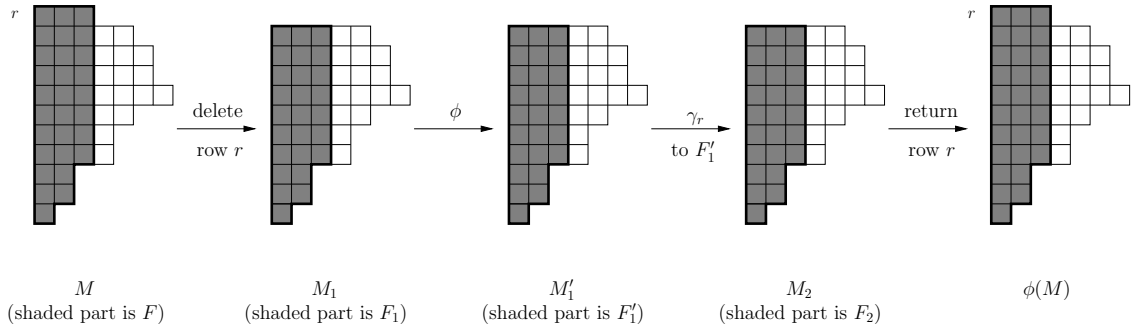}
\end{center}
\caption{Illustration of $\phi$ for left-aligned stack polyominoes.}
\label{stack1}
\end{figure}

It is a bijection from $\n(\M, \s; A)$ to itself since every step is invertible. 
That the step 3 is invertible follows from the proof of Theorem \ref{ferrer}. 

\begin{lemma} \label{mTF}
Using the notation introduced above,
\begin{equation} \label{mtf1}
\m(M)-\m(M_1)= \m(F) -\m(F_1) =\# \R_2(F_1)
\end{equation}
where $\R_2(F_1)$ is defined with respect to the top row $r$ of $F$ as in Section~\ref{ferrers}.
\end{lemma}

\begin{proof}
Let $S$ be the filling  $M$ restricted to the rows which are not
completely contained in $\F$. From the definition of $maj$, we have
\begin{align}
\m(M)&=\m(F) + \m(S) -\m(F \cap S) \label{one} \\
\m(M_1)&=\m(F_1) + \m(S) -\m(F_1 \cap S) \label{two}
\end{align}
Note that $F \cap S =F_1 \cap S$, hence~\eqref{one} and
~\eqref{two} give the first equality in~\eqref{mtf1}. The second
one is readily checked using Proposition~\ref{majstack} for
major index of stack polyominoes.
\end{proof}

The rest of this subsection is devoted to proving that the map $\phi$ has the
desired  property $\m(M)=\se(\phi(M))$. We  shall need one more lemma.

Let $C$ be a $1$-cell  in a filling $M$ of a left-aligned
stack polyomino. We  say that $C$ is  \textit{maximal}
if there is no 1 in $M$ which is strictly above and weakly to
the left of $C$. Clearly, a maximal cell must be the
highest $1$-cell in its column. For a maximal cell $C$ in
$M$, we use $t(C,M)$ to denote the index of the leftmost column
(the columns are numbered from left to right) to the right of $C$
that contains a $1$-cell which together with $C$ forms an NE chain.
If such a column does not exist set $t(C,M)=\infty$.

\begin{lemma} \label{maxcell}
A $1$-cell $C$ in $M$ is maximal if and only if the highest 
$1$-cell $C'$ in the same column in $M'=\phi(M)$ exists and is maximal. 
Moreover, $t(C,M)=t(C',M')$
when they are both maximal. 
\end{lemma}

\begin{proof}
We proceed by induction on the number of rows of $M$ and keep the notation as 
in Figure~11. The case
when $M$ has one row is trivial.

Suppose that the top row $r$ of $M$
has a $1$-cell $C^*$ in the column $i^*$. The cell $C^*$ is
maximal in both $M$ and $M'$ and $t(C^*,M)=\infty= t(C^*,M')$. 
Note that a $1$-cell $C \neq C^*$ in $M$ (resp. $M'$)
is maximal if and only if $C$ is in column $i < i^*$ and is
maximal in $M_1$ (resp. $M_2$).

Denote by $C_1$, $C_1'$, and $C_2$ the highest $1$-cells in
column $i < i^*$ in the fillings $M_1$, $M_1'=\phi(M_1)$, and $M_2$,
respectively. By the induction hypothesis, $C_1$ is maximal in $M_1$
if and only if $C_1'$ is maximal in $M_1'$. $M_2$ is obtained 
from $M_1'$ by
applying the algorithm $\gamma_r$ with respect to the row $r$ to
$F_1'$. $C_1'$ is a left cell in this algorithm, and since
$\gamma_r$ does not change the relative position of the left $1$-cells,
$C_2$ is maximal in $M_2$.

It only remains to prove that $t(C,M)=t(C',M')$. Suppose
$t(C_1,M_1)=t(C_1',M_1')=a$ (they are equal by the inductive
hypothesis).

If $a < i^*$ then $t(C, M)=a$. In $M_1'$ 
it corresponds to  a $1$-cell $D'$ in a column left of
$C^*$ such that $C_1', D'$ form an NE chain. 
Since the algorithm $\gamma_r$ preserves the column sums and the relative 
positions of left $1$-cells, we have $t(C_2,M_2)=a$, and hence $t(C',M')=a$.

If $a \geq i^*$ then $t(C,M)=i^*$. The fact that $t(C_1',M_1')=a
\geq i^*$ means that there is no left $1$-cell in $M_1'$ above and to
the right of $C_1'$ and, since $\gamma_r$ does not change the
relative positions of the left $1$-cells, there will be no left $1$-cell
above and to the right of $C_2$ in $M_2$. So, in this case, we
conclude $t(C',M')=i^*=t(C,M)$.
\end{proof}

\begin{theorem} \label{thmphi} If $M$ is a filling of a left-aligned stack
polyomino $\M$, then $\m(M)=\se(\phi(M))$.
\end{theorem}

\begin{proof}
Again we proceed by induction on the number of rows of $M$. The case
when $M$ has one row is trivial. We use the notation as in Figure 11, and 
let $F_2=M_2 \cap \F_1$.

By Proposition~\ref{alg}, the algorithm $\gamma_r$ on $M_1'$ decreases ne by one for each
left $1$-cell in $\R_1(M_1')$ and increases it by one for each right
$1$-cell in $\R_2(M_1')$. Therefore,
\begin{align*}
\se(\phi(M))&=\se(M_2)+\#\{\text{left $1$-cells in }\R(F_2)\} \\
&=\se(M_1') - \#\{\text{left $1$-cells in } \R_1(F_1')\}+\#\{\text{right
$1$-cells in }\R_2(F_1')\} \\
& \mbox{ \ \ }+\#\{\text{left $1$-cells in }\R(F_2)\}
\end{align*}
Since $\gamma_r$ preserves the column sums, \[\#\{\text{left $1$-cells in }\R(F_2)\}=\#\{\text{left $1$-cells
in }\R(F_1')\}\] hence
\[\se(\phi(M))=\se(M_1')+\#\R_2(F_1').\]

One the other hand, from Lemma~\ref{mTF} and the induction
hypothesis, one gets
\[\m(M)=\m(M_1)+\# \R_2(F_1) = \se(M_1')+ \#\R_2(F_1).\]
So, it suffices to show that

\begin{equation} \label{r} \#\R_2(F_1)=\#\R_2(F_1').
\end{equation}

 To show equation \eqref{r}, note that
for a filling of a Ferrers diagram $F$ the number of rows in $\R_2(F)$ is 
determined by the column indices of $CLC$ (critical left cell) and
$CRC$ (critical right cell). The $CLC$ is the topmost left $1$-cell which, if exists, is the topmost \emph{maximal} cell besides the $1$-cell in the first row. 
By Lemma \ref{maxcell}, $CLC$ in 
$F_1 \subseteq M_1$ is in the same column as $CLC'$ in $F_1' \subseteq M_1'$, 
or neither of them exists. The latter case is trivial, as $\R_2(F)$ is empty. In the first case, 
 let $|r|$ denote the length of row $r$. Then 
 $CRC$,  the critical right cell of $F_1$, is in the column
$t(CLC,M_1)$ if $t(CLC,M_1) \leq |r|$ and does not exist
otherwise. Similarly, the critical right cell of $F_1'$, 
 $CRC'$, is in the column $t(CLC',M_1')$
if $t(CLC,M_1') \leq |r|$, and does not exist otherwise. By
Lemma~\ref{maxcell}, $t(CLC,M_1)=t(CLC',M_1')$, and this
implies~\eqref{r}.
\end{proof}

\subsection{ The case of a general moon polyomino}

Now we consider the case when $\M$ is a general moon polyomino. We label the rows
of $\M$ by $r_1, \dots, r_n$ from top to bottom. Let $\N$ be the
unique left-aligned stack polyomino whose sequence of row lengths
is equal to $|r_1|, \dots, |r_n|$ from top to bottom. In other
words, $\N$ is the left-aligned polyomino obtained by rearranging
the columns of $\M$ by length in weakly decreasing order from left
to right. For the definitions that follow, we use an idea of M.~Rubey~\cite{rubey} to
describe an algorithm that rearranges the columns of $\M$ to obtain
$\N$ (Figure~\ref{alpha}). 

\begin{center}
\begin{tabular}[b]{|p{15cm}|}
\hline
\underline{Algorithm $\alpha$ for rearranging $\M$:}\\
\\
1. Set $\M'=\M$.\\
2. If $\M'$ is left aligned go to (4).\\
3. If $\M'$ is not left-aligned consider the largest rectangle
$\R$ completely contained in $\M'$ that contains $c_1$, the leftmost
column of $\M'$. Update $\M'$ by letting $\M'$ be the polyomino obtained by moving the leftmost column of $\R$ to the right end. Go to (2).\\
4. Set $\N=\M'$.\\ \hline
\end{tabular}
\end{center}

\begin{figure}[ht]
\begin{center}
\includegraphics[width=15cm]{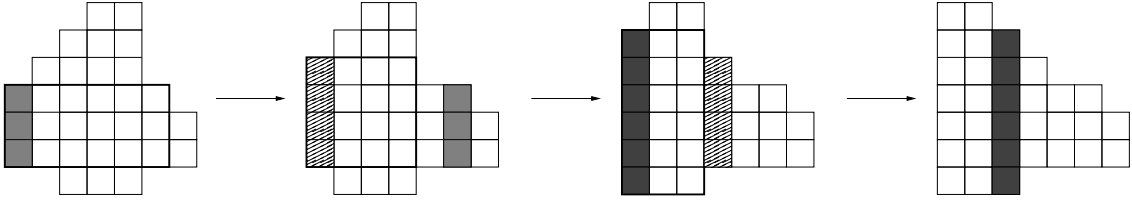}
\end{center}
\caption{The algorithm $\alpha$.}
\label{alpha}
\end{figure}

In this section we give two bijections
$f,g: \n(\M,\s; A) \rightarrow \n(\N,\s';A)$ which preserve the major
index and the number of NE chains, respectively. The sequence
$\s'$ is obtained by rearranging the sequence $\s$ in the same way
$\N$ is obtained by rearranging the columns of $\M$.

\subsubsection{Bijection $f:\n(\M,\s;A) \rightarrow \n(\N,\s';A)$ such that
$\m(M)=\m(f(M))$}

Let $R$ be a filling  of a rectangle $\R$ with column sums $s_1, s_2,
\dots, s_m$. First, we describe a transformation $\tau$ which
gives a filling of $\R$ with column sums $s_2, s_3, \dots, s_m,
s_1$ and preserves the major index. Recall that a descent in $R$ is a pair of 1's in consecutive nonempty rows that form an NE chain. Define an \textit{ascent} to be a pair of 1's in two consecutive nonempty rows that does not form an NE chain. Suppose the first column of $R$ has
$k$ many 1's: $C_1, \dots, C_k$ from top to bottom.
\begin{figure}[ht]
\begin{center}
\includegraphics[width=15cm]{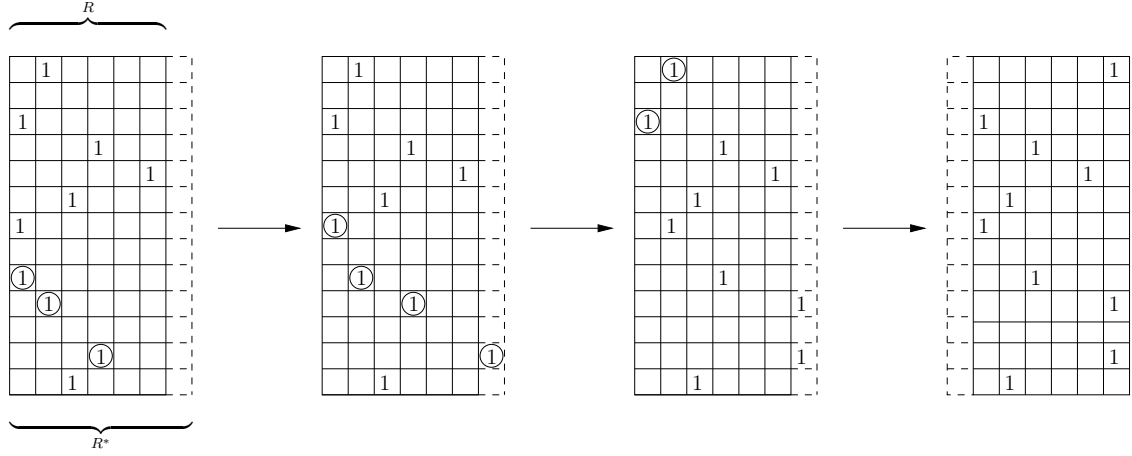}
\end{center}
\caption{Illustration of the transformation $\tau$ for rectangles.}
\label{tau}
\end{figure}

\underline{Transformation $\tau$ on rectangles:}
\begin{enumerate}
\item Let $R^*$ be the rectangle obtained by adding one empty column to
$R$ from right.  Process the $1$-cells $C_1, \dots, C_k$ from
bottom to top (see Figure~\ref{tau}):

\item For $r=k,k-1, \dots, 1$ do the following:

Let $D_a$ be the lowest $1$-cell above $C_r$ in $R^*$ and $D_b$ be
the highest $1$-cell below $C_r$ in $R^*$.

\begin{enumerate}
 \item If $D_a$ does not exist or $D_a=C_{r-1}$,  and $D_b$ does not exist, just move $C_r$
 horizontally to the last column of $R^*$;

 \item If (i) $D_a$ does not exist or $D_a=C_{r-1}$,  but $D_b$ exists, or (ii) both $D_a$ and
$D_b$ exist, and  $D_a$ is to the right of $D_b$:

Let $D_1=C_r, D_2=D_b, D_3,\dots, D_p$ be the maximal chain of consecutive ascents in $R^*$.
Move $D_i$ horizontally to the
column of $D_{i+1}$, for $i=1,\dots, p-1$, and move $D_p$
horizontally to the last column of $R^*$. (If $D_p$ is in the last column of $R^*$, it remains there.)

\item  If $D_a \neq C_{r-1}$,  and ($D_a$ is weakly to the left of $D_b$  or $D_b$ does not exist):

Let $D_1=C_r, D_2=D_a, D_3,
\dots, D_p$ be the maximal chain of consecutive descents in $R^*$.
Move $D_i$ horizontally to the column of $D_{i+1}$, for
$i=1,\dots, p-1$, and move $D_p$ horizontally to the last column
of $R^*$.
\end{enumerate}

\item Delete the first column of $R^*$.
\end{enumerate}

Note that each iteration in Step 2 decreases the number of 1's in the first
column of $R^*$ by one and increases the number of 1's in the last
column by one, while the other column sums remain unchanged.

Now we are ready to define the map $f$. Suppose $M \in \n(\M,\s; A)$.
We perform the algorithm $\alpha$ on $M$ to transform the shape $\M$ to
$\N$. While we are in Step 3, instead of just moving the first column of $\R$ to the right end, we perform the algorithm $\tau$ on the filling $R$ of $\R$. $f(M)$ is defined to be the resulting filling of $\N$.

\begin{lemma} \label{lemmatau} The transformation $\tau$ is invertible. Moreover, it preserves the descents and hence $\m(R)=\m(\tau(R))$.
\end{lemma}

\begin{proof} The map $\tau$ is invertible since  $\tau^{-1}$ can be obtained by taking
$\rho \circ \tau \circ \rho$, where $\rho$ is the rotation of rectangles by $180^\circ$. 
Moreover, each iteration of Step 2 of $\tau$ preserves the positions of descents, hence the 
second part of the claim holds.
\end{proof}

\begin{proposition} \label{propf} The map $f:\n(\M,\s; A) \rightarrow \n(\N,\s'; A)$ is a bijection and $\m(M)=\m(f(M))$.
\end{proposition}

\begin{proof} The first part follows from Lemma~\ref{lemmatau} and the fact that $\alpha$ is invertible. The filling $f(M)$ is obtained after several iterations of Step 3 in the algorithm $\alpha$ combined with application of $\tau$. For the second part of the claim, it suffices to show that  maj is preserved after one such iteration. Let $M'$ be  the filling of the shape $\M'$ obtained from $M$ after one step of $f$ in which the rectangle $R$ is transformed into $\tau(R)$.  We shall use Definition 2.2 to compare $\m(M)$ and $\m(M')$.
Let $R_1, \dots, R_k$ (resp. $R_1', \dots, R_k'$) be the maximal rectangles of $M$ (resp. $M'$), 
ordered by height. In particular, $R_i$ and $R_i'$ are of same size.  Suppose $R=R_s$, then $R_s'=\tau(R)$. 
We partition the index $i$ into three intervals, according to $i< s$, $i=s$, and $i >s$, 
and compare the contribution of $R_i$'s in each interval.

\begin{enumerate}
 \item For $1 \leq i \leq s-1$, the rectangles $R_i$ and  $R_i \cap R_{i+1}$ are of 
smaller height than $R$ and contain several consecutive rows of $R$. Using Lemma~\ref{lemmatau}
we know that the descents in $R_i'\cap \tau(R)$ (resp. $R_i' \cap R_{i+1}'\cap \tau(R)$) 
$1 \leq i \leq s-1$, are in the same rows as in $R_i\cap R$ (resp. $R_i \cap R_{i+1}\cap R$). And descents formed by one cell in $R$ and one cell on the right of $R$ are preserved since 
$\tau$ preserves the row sum of $R$. 
Consequently,
\begin{equation} \label{f1}
\m(R_i)=\m(R_i'),\; \m(R_i \cap R_{i+1})=\m(R_i' \cap R_{i+1}'),\; 1 \leq i \leq s-1.
\end{equation}

\item For $i=s$, by Lemma~\ref{lemmatau}, we have
\begin{equation} \label{f2}
\m(R_s)=\m(R_s').
\end{equation}

\item For $i > s$, 
the height of the rectangle $R_i$ or $R_{i-1} \cap R_{i}$ is greater than or equal to that of $R$. 
Each of these rectangles contains several consecutive columns of $R$ excluding the first one. 
The filling of each subrectangle of $R$ consisting of the columns $c_j, \dots, c_{j+m}$ ($j>1$) is equal, up to a rearrangement of the empty rows, to the filling of the subrectangle of $\tau(R)$ consisting of the columns $c_{j-1}, \dots, c_{j+m-1}$. To see this, note that the shifting of the $1$-cells $D_1, \cdots, D_p$ horizontally in the algorithm $\tau$ can be viewed as moving the 
$1$-cell $D_i$ vertically to the row of $D_{i-1}$ and moving $D_1$ to the last column of $R^*$ in the row of $D_p$. Thus, the descent positions in the fillings  $R_i'\cap \tau(R)$ (resp. $R_{i-1}' \cap R_{i}'\cap \tau(R)$), $i > s$ are the same as in $R_i\cap R$ (resp. $R_{i-1} \cap R_{i}\cap R$). 
Finally, one checks that the descents formed by one cell in $R$ and one cell outside $R$ are 
also preserved. Therefore,
\begin{equation} \label{f3}
\m(R_i)=\m(R_i'),\; \m(R_{i-1} \cap R_{i})=\m(R_{i-1}' \cap R_{i}'),\; i >s.
\end{equation} 
\end{enumerate}
Combining~\eqref{f1},~\eqref{f2}, and~\eqref{f3} gives $\m(M)=\m(M')$.
\end{proof}

\subsubsection{Bijection $g:\n(\M,\s; A) \rightarrow \n(\N,\s'; A)$ such that
$\se(M)=\se(g(M))$}

Suppose $M \in \n(\M,\s; A)$. To obtain $g(M)$, we perform the algorithm $\alpha$ to
transform the shape $\M$ to $\N$ and change the
filling when we move columns in Step 3 so that the number of 1's
in each row and column is preserved. 

Let $R$ be the rectangular filling in Step 3 of $\alpha$ that contains the column
$c_1$ of the current filling. Suppose $c_1$ contains $k$ many
$1$-cells $\cone, \dots, C_k$ from top to bottom. Shade the
empty rows of $R$ and the cells in $R$ to the right of
$\cone, \dots, C_k$. Let $l_i$ denote the number of empty white cells
in $R$ above $C_i$. If $\R'$ is the rectangle obtained by
moving the column $c_1$ from first to last place, fill it to obtain a filling $R'$ as follows.

\begin{enumerate} 
\item The rows that were empty remain empty. Shade these rows.

\item Write $k$ many 1's in the last column from bottom to top so that
there are $l_i$ white empty cells below the $i$-th 1, $1
\leq i \leq k$.

\item Shade the cells to the left of the nonempty cells in the
last column.

\item Fill in the rest of the rectangle  by writing 1's in the unshaded
rows of $\R'$ so that the unshaded part of $R$ to the right of the first column is the same as the unshaded part of $R'$ to the left of the last column.

\end{enumerate} 
See Figure~\ref{mapg} for an example.

\begin{figure}[ht]
\begin{center}
\includegraphics[width=15cm]{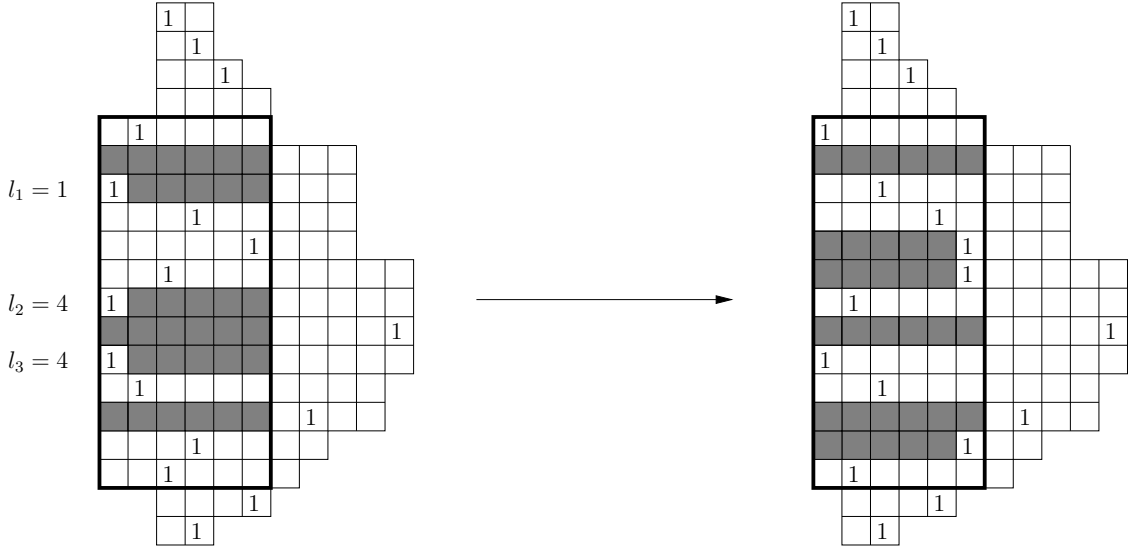}
\end{center}
\caption{One step of the map $g$.}
\label{mapg}
\end{figure}

\begin{proposition} \label{propg} The map $g:\n(\M,\s;A) \rightarrow \n(\N,\s';A)$ is a bijection and
$\se(M)=\se(g(M))$.
\end{proposition}

\begin{proof}
Clearly, every step of $g$ is invertible. To see that ne is preserved, it suffices to show that it is preserved in each step of $g$, when one column is moved from left to right. First note that $\se(R)=\se(R')$. This follows from the fact that 
the 1's in the first column of $R$  form $l_1, \dots, l_k$ NE chains from top to bottom, while 
the 1's in the last  column of $R'$ form $l_1, \dots, l_k$ NE chains from bottom to top and the remaining parts of $R$ and $R'$ are essentially equal. The numbers of NE chains in $M$ and $M'$ respectively made of at least one 1 which is outside of $R$ and $R'$ respectively are equal because each step of $g$ preserves the row and column sums.
\end{proof}

\begin{theorem}
Let $\phi$ denote the Foata-type bijection described in
Section~\ref{stack}. For a general moon polyomino $M$ the map
$\psi = g^{-1} \circ \phi \circ f:\n(\M,\s;A) \rightarrow
\n(\M,\s;A)$ is a bijection with the property
\begin{equation}\label{final}\m(M)= \se(\psi(M)) \end{equation}
\end{theorem}

\begin{proof}
The maps $f$ and $g$ permute the columns of $M$ together with the
corresponding number of 1's, while $\phi$ preserves the column
sums. Moreover, all three of them preserve the empty rows.
Therefore, $\psi$ is indeed a map from $\n(\M,\s;A)$ to itself.
Finally,~\eqref{final}  follows from Proposition~\ref{propf}, Proposition~\ref{propg}, and Theorem~\ref{thmphi}.
\end{proof}


\section*{Acknowledgment}

The authors would like to thank Ira Gessel~\cite{gessel} for suggesting a connection between the enumeration of permutations by major index and the theory of $P$-partitions.


\section*{Appendix}

To show that the map $\phi: \n(\F,\s;A) \rightarrow \n(\F,\s;A)$ is a bijection, it suffices to describe the inverse of the
algorithm $\gamma_r$. First, we need to determine what $\R_1$ and $\R_2$ were. For that we use the following properties of $\gamma_r$.

\begin{enumerate}

\item \label{r21} When $\gamma_r^1$ stops, the pointer $\pone$ is at
the lowest row of $\R_1$ pointing at a right $1$-cell $C_1$. The
algorithm $\gamma_r^2$ does not change the cell $C_1$ and when the
whole algorithm $\gamma_r$ terminates, there is no $1$-cell below $\cone$
that together with it forms an NE chain.

\item \label{r22} After performing $\gamma_r^2$, all the right $1$-cells
in $\R_2$ have at least one $1$-cell below them with which they form an
NE chain.

\item If $\R_2$ was nonempty then, after performing $\gamma_r$, the
lowest row of $\R_2$ contains a left $1$-cell.

\item \label{borrowing} If any borrowing occurred and the right cell
$\cone=(i_1,j_1)$ and the left cell $\ctwo=(i_2,j_2)$ were
swapped, then right after this step the cell $(i_1,j_2)$ is a
right $1$-cell in $\R_1$ which forms an NE chain with a left $1$-cell
from $\R_2$ and this is the lowest $1$-cell in $\R_1$ with this
property.

\end{enumerate}

Therefore, if $\R$ has no left $1$-cells then $\R_2=\R_1=\emptyset$. Otherwise,
find the lowest right $1$-cell $C^*$ in $\R$ such that there is no
$1$-cell in $\R$ below it that together with $C^*$ forms an NE chain.
Then, by Properties~\ref{r21} and~\ref{r22} all the nonempty rows
in $\R$ below $C^*$ are in $\R_2$. If there is no $1$-cell $C^*$ with
that property, then $\R_2$=$\R$.

\subsection*{\textbf{Algorithm $\delta_r^2$: the inverse of $\gamma_r^2$ }}

\begin{enumerate}
\item [\textsf{\textbf{(IA$'$)}}] \textsf{Initially, set $\pone$ to  the
lowest row of $\R_2$ and $\ptwo$ to the next row in $\R_2$ above
$\pone$.}
\item [\textsf{\textbf{(IB$'$)}}] \textsf{If $\ptwo$ is null then go to step
(ID$'$). Otherwise, suppose $\pone$ and $\ptwo$ are pointing at
$1$-cells $\cone$ and $\ctwo$.}
\textsf{
\begin{enumerate}
\item [\textbf{(IB$'$1)}] If $\ctwo$ is a left cell do nothing.
\item [\textbf{(IB$'$2)}] If $\ctwo$ is a right cell then swap the
cells $\cone$ and $\ctwo$.
\end{enumerate}}
\item [\textsf{\textbf{(IC$'$)}}] \textsf{Move  $\pone$ to the row of $\ptwo$ and
$\ptwo$ to the next row in $\R_2$ above it. Go to (IB$'$).}

\item [\textsf{\textbf{(ID$'$)}}] (\textit{\textbf{Inverse borrowing}})\textsf{ Note
that $\pone$ always points to a left $1$-cell. When it reaches the
highest row of $\R_2$ and points to a $1$-cell $\cone$ first we need
to determine whether there was any borrowing. For that purpose
find the lowest right $1$-cell $\ctwo$ above $\cone$ such that the two
cells form an NE chain. Using Property~\ref{borrowing} we
conclude:}
\textsf{
\begin{enumerate}
\item [\textbf{(ID$'$1)}] If there is no such a cell $\ctwo$, then there was
no borrowing so do nothing.
\item [\textbf{(ID$'$2)}] If there is such a cell $\ctwo$, and there is a
left $1$-cell in the rows between $\cone$ and $\ctwo$, then there was
no borrowing so do nothing.
\item [\textbf{(ID$'$3)}] If there is such a cell $\ctwo$, and there is
no left $1$-cell between $\cone$ and $\ctwo$, then swap $\cone$ and $\ctwo$.
\end{enumerate}}
\item [\textsf{\textbf{(IE$'$)}}] \textsf{Stop.}

\end{enumerate}

When the algorithm $\delta_r^2$ stops, continue by applying the algorithm $\delta_r^1$ on $\R \backslash \R_2$.

\subsection*{\textbf{Algorithm $\delta_r^1$: the inverse of  $\gamma_r^1$ }}

\begin{enumerate}
\item [\textsf{\textbf{(IA)}}]\textsf{ Position $\pone$ on the lowest row in $\R
\backslash \R_2$ and $\ptwo$ on the lowest nonempty row in $\R
\backslash \R_2$ above it.
\item [\textbf{(IB)}] If $\ptwo$ is null then go to step (ID).
Otherwise, suppose that $\pone$ and $\ptwo$ point at the 
$1$-cells $\cone$ and $\ctwo$ respectively.}

\begin{enumerate}
\item [\textbf{(IB1)}] \textsf{If $\ctwo$ is a left cell then check
whether there exists a right $1$-cell $R$ above $\ctwo$ in a column
longer than the column of $\cone$ such that there are no left $1$-cells between $R$ and $\ctwo$.}

\textsf{
\begin{enumerate}
\item [\textbf{(IB1.1)}]  If such a cell $R$ exists, then suppose that the  row-column
coordinates of $\cone$, $\ctwo$, and $R$ are $(i_1,j_1)$, $(i_2,j_2)$, and $(i_3,j_3)$,
respectively. Delete the 1's from these three cells and write them
in the cells with coordinates $(i_1,j_3)$, $(i_2,j_1)$, and
$(i_3,j_2)$. Move $\pone$ to the row of $\ptwo$.
\item [\textbf{(IB1.1)}] If there is no such a cell $R$, 
then swap $\cone$ and $\ctwo$ and
move $\pone$ to the row of $\ptwo$.
\end{enumerate}}

\item [\textbf{(IB2)}] \textsf{If $\ctwo$ is a right cell, then
\begin{enumerate}
\item [\textbf{(IB2.1)}] If  $|col(\cone)|=|col(\ctwo)|$ 
then move $\pone$ to the row of $\ptwo$.
\item [\textbf{(IB2.2)}] If $|col(\cone)| \neq |col(\ctwo)|$ then do nothing.
\end{enumerate} }
\end{enumerate}
\textsf{
\item [\textbf{(IC)}] Move $\ptwo$ to the next row in $\R \backslash \R_2$ above
it. Go to (IB).
\item [\textbf{(ID)}] Stop.} 

\end{enumerate}

Suppose $M'$ is a filling obtained by applying the algorithm $\gamma_r$ to the filling $M \in \n(\M,\s;A)$. Next we show that by applying $\delta_r^2$ and $\delta_r^1$ to $M'$ one obtains the filling $M$.

Steps (IB$'$1) and (IB$'$2) clearly invert the steps (C$'$2) and (C$'$1), respectively. Immediately after the borrowing step in $\gamma_r^2$, by Lemma~\ref{prop}, (b) and (c), and Property~\ref{borrowing}, we can conclude that borrowing has been performed if and only if there is a right $1$-cell in $\R \backslash \R_2$ which forms an NE chain with a left $1$-cell 
from $\R_2$ and does not form an NE chain with any left cells in $\R_1$. It is clear that step (ID$'$) detects whether there was any borrowing and inverts it if it occurred. So, the first algorithm $\delta_r^2$ is indeed the inverse of $\gamma_r^2$. 

In $\delta_r^1$, when $\ptwo$ points to a left $1$-cell we need to invert either (B1) or (B2.3). Using  Lemma~\ref{prop} (c) one sees that step (IB1) detects exactly which of (B1) and (B2.3) occurred and inverts it. If $\ptwo$ points to a right cell, it is clear  that (IB2.1) and (IB2.2) invert the steps (B2.1) and (B2.2) of $\gamma_r^1$. Moreover, if both $\pone$ and $\ptwo$ point 
to right $1$-cells, and there is no left $1$-cell above them, then clearly the algorithm 
$\delta_r^1$ leaves them unchanged. That is, the algorithm $\delta_r^1$ preserves all rows that 
are above the rows in $\R_1$.  Therefore, we have
\[ M \overset{\gamma_r}{\longrightarrow} M' \overset{\text{   }\delta_r^2 + \delta_r^1\text{   }}{\longrightarrow} M.\]
This implies that $\gamma_r:\n(\M,\s;A) \rightarrow \n(\M,\s;A)$ is injective. Since $\n(\M,\s;A)$ is finite, it follows that $\gamma_r$ is bijective.

\end{document}